\documentclass[a4paper]{article}

\usepackage[utf8]{inputenc}
\usepackage[T1]{fontenc}
\usepackage{amsfonts}

\usepackage{amsmath}
\usepackage{amssymb}
\usepackage{amsthm}
\usepackage{latexsym}
\usepackage{bm}

\usepackage[fixamsmath]{mathtools}
\mathtoolsset{showmanualtags,mathic,centercolon}
\usepackage{enumitem}
\usepackage[svgnames]{xcolor}
\usepackage{algpseudocode}
\usepackage{algorithm}
\usepackage{url}
\usepackage{graphicx}
\usepackage[margin=1.25in]{geometry}
\usepackage[labelfont=sl,textfont=sl]{subcaption}
\usepackage[font={small,sl},labelsep=period]{caption}

\numberwithin{equation}{section}
\theoremstyle{plain}
\newtheorem{theorem}{Theorem}[section]
\newtheorem{lemma}[theorem]{Lemma}
\newtheorem{corollary}[theorem]{Corollary}

\theoremstyle{definition}
\newtheorem{definition}[theorem]{Definition}
\newtheorem{example}[theorem]{Example}
\newtheorem{assumption}[theorem]{Assumption}

\theoremstyle{remark}

\newcommand{\R}{\mathbb{R}}

\newcommand{\bigO}{\mathcal{O}}

\newcommand{\grad}{\nabla}

\newcommand{\calD}{\mathcal{D}}

\renewcommand{\b}[1]{\bm{#1}} 
\newcommand{\ba}{\b{a}}

\newcommand{\bc}{\b{c}}
\newcommand{\bd}{\b{d}}
\newcommand{\be}{\b{e}}
\newcommand{\bg}{\b{g}}

\newcommand{\br}{\b{r}}

\newcommand{\bu}{\b{u}}
\newcommand{\bv}{\b{v}}
\newcommand{\bx}{\b{x}}
\newcommand{\by}{\b{y}}

\newcommand{\bA}{\b{A}}

\DeclareMathOperator*{\proj}{proj}
\newcommand{\gammainc}{\gamma_{\textnormal{inc}}}
\newcommand{\gammadec}{\gamma_{\textnormal{dec}}}
\newcommand{\flow}{f_{\textnormal{low}}}
\newcommand{\nb}{p}
\DeclareMathOperator*{\cm}{cm}
\newcommand{\cdec}{\sigma}
\DeclareMathOperator*{\col}{col}

\begin{document}
\title{Polling Set Construction and Worst-Case Complexity for Direct Search under Polyhedral Convex Constraints}
\author{
Lindon Roberts\thanks{School of Mathematics and Statistics \& ARC Training Centre in Optimisation Technologies, Integrated Methodologies, and Applications (OPTIMA), University of Melbourne, Parkville VIC 3010, Australia. ORCID 0000-0001-6438-9703 (\texttt{lindon.roberts@unimelb.edu.au}). 
This author was supported by the Australian Research Council Discovery Early Career Award DE240100006, 
and by CNRS IAE under the grant BONUS.
}
\and 
Cl\'ement W. Royer
\thanks{LAMSADE, CNRS, Universit\'e Paris Dauphine-PSL, Place du Mar\'echal de Lattre de Tassigny, 75016 Paris, France. ORCID 0000-0003-2452-2172 (\texttt{clement.royer@lamsade.dauphine.fr}).
This author was partially supported by Agence Nationale de la Recherche 
through program ANR-23-IACL-0008 (PR[AI]RIE-PSAI), 
and by CNRS IAE under the grant BONUS.
}}

\date{\today}
\maketitle

\begin{abstract}
Direct search is one of the most popular derivative-free optimization 
paradigms, that relies on exploring the variable space using polling 
directions. To analyze and implement direct search, one typically relies 
on positive spanning sets. This concept is somewhat decorrelated from 
interpolation-based sets used in model-based algorithms, another class of 
derivative-free optimization methods. This discrepancy is even more 
pronounced for constrained problems, where recent advances in the 
interpolation-based setting have produced a unified picture that is 
still lacking in the direct-search case.

In this paper, we introduce a new theoretical underpinning for 
direct-search methods, that can be defined for general convex 
constraints and lead to complexity guarantees, as in the model-based 
setting. By focusing on polyhedral convex constraints, we are able to 
construct polling sets that meet our new theoretical requirements. In 
particular, our polling sets necessarily include directions outside of the 
approximate tangent cone, giving theoretical justification to existing 
practical heuristics which incorporate this 
idea. Our numerical results confirm that adding these extra directions 
significantly improves practical performance.
\end{abstract}




\section{Introduction}
\label{sec_intro}

Optimization problems that arise in practice commonly feature black-box, 
computationally expensive, or noisy objectives~\cite{Alarie2021}. As a result, 
standard derivative-based algorithms cannot be used to tackle these problems, 
but derivative-free optimization (DFO) methods form a useful 
alternative~\cite{Audet2025,Conn2009,Larson2019}. 
%
Among DFO methods, \emph{direct search} schemes are quite popular because of 
their simplicity of description and analysis~\cite{Dzahini2025,Kolda2003}, 
and have given rise to efficient practical techniques~\cite{Audet2022}. Each 
iteration of a direct-search methods builds a finite \emph{polling set} of 
(feasible) points near the current iterate. The quality of polling sets is 
instrumental to obtaining convergence guarantees for direct-search algorithms. 
In particular, when the problem is unconstrained, convergence is ensured by 
using \emph{positive spanning sets} (PSSs), that span the variable space 
through nonnegative linear combinations, as polling sets. Moreover, by 
ensuring those sets have good enough \emph{cosine measure}, worst-case 
complexity results can be established~\cite{Vicente2013}. Although using 
random directions has recently gained popularity as a scalable alternative 
to positive spanning sets~\cite{Gratton2015,Roberts2023}, the use of the 
latter remains necessary to guarantee deterministic convergence.

Choosing polling sets in presence of constraints is a significant challenge, 
as the feasible set geometry must be taken in account. To the best of our 
knowledge, and despite recent global convergence results on this 
topic~\cite{XJia_MLapucci_PMansueto_2025}, there is no direct-search technique 
equipped with complexity guarantees for convexly constrained problems. Nevertheless, a 
number of approaches have been developed to tackle the linearly constrained 
case~\cite{Abramson2008,Audet2015,Kolda2003,Kolda2007,Lewis2000}, with bound 
constraints being the subject of dedicated 
investigation~\cite{Brilli2024,Lewis1999,Lucidi2002}. In the bound-constrained 
setting, using coordinate directions and their negative as polling sets gives 
rise to feasible iterate methods with convergence 
guarantees~\cite{Lewis1999,Lucidi2002}. This approach can be extended to 
general linear inequality constraints, by requiring polling sets to generate 
approximate tangent cones defined by nearby constraints, using nonnegative 
linear combinations as in the case of PSSs~\cite{Kolda2007,
Lewis2000}. Since approximate tangent cones can only be described using finite 
generators when they are polyhedral (e.g.~\cite[Theorem 2.8.8]{Stoer1970}), 
this approach does not extend to the convexly constrained case.
Still, these polling choices not only lead to global convergence results, 
but also allow for a worst-case complexity analysis the in presence of general 
linear constraints~\cite{Gratton2019}. 
However, in practice, the performance of these algorithms is greatly improved 
by searching along directions defined by the constraint 
normals~\cite{Lewis2007}. Thus, even though the concept of PSS has been used 
successfully to obtain theoretically sound direct-search techniques in the 
linearly constrained setting, the practical implementations of these methods 
has yet to be fully understood.

Meanwhile, another class of DFO algorithms termed model-based algorithms, 
proceeds by building models of the objective function, that 
can then be minimized (typically over a trust region) in order to produce a 
candidate for the next function evaluation~\cite{Conn2009}. 
Model-based DFO methods consider interpolation sets at which function 
values are available to construct such models~\cite{Roberts2025}. Provided 
those sets satisfy a geometric quality condition termed $\Lambda$-poisedness, 
the resulting algorithms can be endowed with 
complexity analysis in both unconstrained and convexly constrained 
settings~\cite{Hough2022,Roberts2025}. Unlike in direct-search methods, the 
use of approximate tangent cones is not required in model-based DFO, 
and thus the notion of $\Lambda$-poisedness readily extends to the convex 
setting~\cite{Roberts2025convex}. Overall, despite both model-based 
and direct-search schemes being developed with geometrical insights, the 
connection between polling directions and interpolation sets remains 
underexplored. 

In this paper, we introduce the notion of $\Lambda$-positive spanning sets ($\Lambda$-PSS), 
a novel property for polling sets in direct search inspired by $\Lambda$-poisedness in 
model-based derivative-free methods. In the unconstrained setting, our 
definition is similar to that of positive spanning sets, and leads to 
comparable complexity guarantees. For linearly constrained problems, our 
$\Lambda$-PSS theory relies on using polling directions lying outside of approximate 
tangent cones, thus departing from existing theory, but not standard practice. 
In fact, our explicit constructions of $\Lambda$-PSS provide a theoretical 
grounding for popular implementations~\cite{Gratton2019,Lewis2007}. 
More broadly, our numerical experiments demonstrate that using generators 
of approximate tangent cones and their negatives (a valid $\Lambda$-PSS 
construction) outperforms both classical approaches with theoretical 
guarantees, that only rely on approximate tangent cones, and standard 
practice, which leverages active constraint normals.

The rest of this paper is organized as follows. 
Section~\ref{sec_background} provides background material on direct-search 
methods for smooth unconstrained and linear inequality constrained problems. 
Section~\ref{sec_lbda_pss} introduces a new approach for measuring the 
quality of polling sets, which is then used in Section~\ref{sec_newwcc} to 
revisit complexity guarantees for direct search on both unconstrained and 
(polyhedral) convex constrained problems. Section~\ref{sec_poll_sets} 
provides detailed constructions for polling sets in presence of explicit 
linear constraints. The numerical performance of this approach is 
showcased in Section~\ref{sec_numerics}.

\paragraph{Notation}
Throughout the paper, we use $\|\cdot\|$ to be the Euclidean norm of vectors 
and operator 2-norm (i.e.~largest singular value) of matrices. Given a matrix 
$\bA$, we denote its Moore-Penrose pseudoinverse as $\bA^{\dagger}$ and the 
set of its column vectors by $\col(\bA)$. Given a 
cone $K\subseteq\R^n$, we denote its polar cone as 
$K^{\circ} := \{\by\in\R^n : \by^T \bx \leq 0, \: \forall \bx\in K\}$.
For $\bx\in\R^n$ and $\alpha>0$ we define 
$B(\bx,\alpha) := \{\by\in\R^n : \|\by-\bx\| \leq \alpha\}$ to be the closed 
ball of radius $\alpha$ centered at $\bx$. Finally, the projection onto a 
convex set $S$ will be denoted by $\proj_S[\bv]$.

\section{Linearly constrained problems and direct search} 
\label{sec_background}

In this paper, we study direct-search methods applied to solving problems of 
the form
\begin{align}
    \min_{\bx\in\Omega} \: f(\bx), \label{eq_convex_cons_problem}
\end{align}
where we assume that $f$ is smooth and $\Omega$ is a polyhedral set, in the 
sense of the following two assumptions.

\begin{assumption} \label{ass_smoothness}
    The objective function $f$ \eqref{eq_convex_cons_problem} is $C^1$, its 
    gradient $\grad f$ is $L$-Lipschitz continuous for some $L>0$, and $f$ is 
    bounded below on $\Omega$ by $\flow$.
\end{assumption}

\begin{assumption} \label{ass_omega}
	The feasible set $\Omega$ is a a polyhedral convex set with nonempty 
	interior. Moreover, it admits a description of the form
	\begin{align}
    	\Omega = \{\bx\in\R^n : \ba_i^T \bx \leq b_i, \: \forall i=1,\ldots,m\}.
    \label{eq_linear_cons}
	\end{align}
\end{assumption}

Note that the description~\eqref{eq_linear_cons} covers the unconstrained 
setting ($m=1$, $\ba_1=\mathbf{0}$ and $b_1=0$), but that $\Omega$ cannot 
be a proper subspace of $\R^n$. Linear equality constraints, that define 
linear subspaces, can however be accounted for by other techniques closer to the 
unconstrained setting~\cite{Gratton2019}, and are not the focus of this paper.

The rest of this section describes existing results on direct-search schemes 
applied to problem~\eqref{eq_convex_cons_problem}. In particular, we recall 
existing complexity results in the unconstrained and linearly constrained 
setting, that leverage the concept of positive spanning set (PSS).

\subsection{Direct-search algorithm}
\label{sec_background_dsalgo}

Algorithm~\ref{alg_ds_basic} describes a generic direct-search framework for 
solving problem~\eqref{eq_convex_cons_problem}. The method starts with a 
feasible point and explores the space along directions that preserve 
feasibility. At every iteration, a finite set of $\nb$ polling directions 
$\mathcal{D}_k \subset \R^n$ is selected, and the function value is queried 
at $\bx_k + \alpha_k \bd$ for $\bd\in\mathcal{D}_k$, where $\bx_k$ is the 
current iterate and $\alpha_k$ is a dynamically updated stepsize. If a polled 
point with sufficiently small objective value is found, then this becomes the 
new iterate and the stepsize is usually increased. Otherwise, the iterate does 
not change and the stepsize is reduced.

\begin{algorithm}[tb]
\begin{algorithmic}[1]
\Statex \textbf{Inputs:} $\bx_0 \in \Omega$, $\alpha_{\max} > 0$, 
	$\alpha_0 \in (0,\alpha_{\max}]$, $\cdec>0$, 
    $0 < \gammadec < 1 < \gammainc$.
\vspace{0.2em}
\For{$k=0,1,2,\ldots$}
    \State Compute a polling set $\mathcal{D}_k \subset \R^n$.
    \State If there exists $\bd_k \in \mathcal{D}_k$ such that 
    $\bx_k + \alpha_k \bd_k \in \Omega$ and
    \begin{align}
        f(\bx_k+\alpha_k \bd_k) < f(\bx_k) - \frac{\cdec}{2}\alpha_k^2, 
        \label{eq_sufficient_decrease}
    \end{align}
    set $\bx_{k+1}:=\bx_k+\alpha_k \bd_k$ and 
    $\alpha_{k+1}:=\min\{\gammainc \alpha_k, \alpha_{\max}\}$ 
    (``successful iteration''). 
    \State Otherwise, set $\bx_{k+1}:=\bx_k$ and 
    $\alpha_{k+1}:=\gammadec \alpha_k$ (``unsuccessful iteration'').
\EndFor
\end{algorithmic}
\caption{Basic direct-search method for \eqref{eq_convex_cons_problem}.}
\label{alg_ds_basic}
\end{algorithm}

Algorithm~\ref{alg_ds_basic} relies on a sufficient decrease 
condition~\eqref{eq_sufficient_decrease} to accept trial points, which is 
necessary to obtain complexity guarantees~\cite{Dzahini2025,Vicente2013}. 
However, our condition departs from standard choices in that it does not 
involve $\|\bd_k\|$. Removing the norm dependency turns out to be critical 
for our new polling set technique, as 
will be shown in Section~\ref{sec_newwcc}.

\subsection{Positive spanning sets and the unconstrained case}
\label{sec_background_pss}

In this section, we focus on the unconstrained case, i.e. 
problem~\eqref{eq_convex_cons_problem} with $\Omega=\R^n$. In this setting, 
convergence of Algorithm~\ref{alg_ds_basic} can be guaranteed by ensuring 
that every polling set $\calD_k$ contains at least one descent direction, i.e. 
a direction making an acute angle with the negative gradient. When $\calD_k$ 
is chosen as a \emph{positive spanning set} (for $\R^n$), this property is 
ensured without access to derivative information.

\begin{definition} \label{def_pss}
    A set $\{\bd_1,\ldots,\bd_p\}\subset\R^n$ is a positive spanning set 
    (PSS) for $\R^n$ if, for any $\bv\in\R^n$, there exist constants 
    $c_1,\ldots,c_p \geq 0$ such that $\bv=\sum_{i=1}^{p} c_i \bd_i$.
\end{definition}

\begin{lemma}[Theorem 3.1, \cite{Davis1954}] \label{lem_pss_descent_orig}
    The set $\mathcal{D}:=\{\bd_1,\ldots,\bd_p\}\subset\R^n$ is a PSS for 
    $\R^n$ if and only if 
    \begin{align}
        \min_{\bv\neq\bm{0}} \max_{\bd\in \mathcal{D}} \: \bd^T \bv > 0. 
     \label{eq_strict_descent}
    \end{align}
\end{lemma}

It follows from Lemma~\ref{lem_pss_descent_orig} that a set of (nonzero) 
vectors $\calD$ is a PSS if and only if $\cm(\calD)>0$, where $\cm(\calD)$ 
is the \emph{cosine measure} of $\calD$ defined by
\begin{align}
    \operatorname{cm}(\mathcal{D}) := 
    \min_{\bv\neq\bm{0}} 
    \max_{\substack{\bd\in\mathcal{D} \\ \bd \neq \bm{0}}} 
    \frac{\bd^T \bv}{\|\bd\|\, \|\bv\|}. 
    \label{eq_cosine_measure}
\end{align}

Assuming that the sequence $\{\cm(\calD_k)\}$ is uniformly bounded away 
from zero leads to both convergence and complexity 
guarantees~\cite{Dzahini2025, Vicente2013}.

\begin{assumption} \label{ass_kappa_descent}
    At every iteration of Algorithm~\ref{alg_ds_basic}, the set $
    \mathcal{D}_k$ satisfies $\operatorname{cm}(\mathcal{D}_k) \geq \kappa$, 
    and $d_{\min} \leq \|\bd\| \leq d_{\max}$ for all $\bd\in \mathcal{D}_k$, 
    for some $\kappa>0$ and $d_{\max} \geq d_{\min} > 0$ independent of $k$. 
    Moreover, $\calD_k$ has at most $p$ vectors, where $p$ is a positive 
    integer value.
\end{assumption}

\begin{theorem} \label{thm_ds_wcc_old}
    Suppose Assumptions~\ref{ass_smoothness} and~\ref{ass_kappa_descent} 
    hold. Then Algorithm~\ref{alg_ds_basic} achieves 
    $\|\grad f(\bx_k)\| \leq \epsilon$ after at most 
    $\bigO(\kappa^{-2} \epsilon^{-2})$ iterations, or 
    $\bigO(p\kappa^{-2} \epsilon^{-2})$ objective evaluations.
    Hence $\liminf_{k\to\infty} \|\grad f(\bx_k)\| = 0$.
\end{theorem}

A canonical choice for $\mathcal{D}_k$ is the set of positive and negative 
coordinate vectors, $\mathcal{D}_k = \{\pm \be_1, \ldots, \pm\be_n\}$, which 
satisfies Assumption~\ref{ass_kappa_descent} with $\kappa=1/\sqrt{n}$ and 
$d_{\min}=d_{\max}=1$. Using this polling set at every iteration yields  
complexity guarantees in $\bigO(n\epsilon^{-2})$ iterations and 
$\bigO(n^2 \epsilon^{-2})$ objective evaluations, that are optimal for direct 
search in terms of dependencies on $n$~\cite{Dodangeh2016}.

\subsection{Linearly constrained case} 
\label{sec_background_lincons}

We now turn to the linearly constrained setting. Under 
Assumptions~\ref{ass_smoothness} and~\ref{ass_omega}, first-order stationarity 
can be assessed using the following measure~\cite{Cartis2012,Hough2022}:
\begin{align}
    \pi(\bx) := \left|\min_{\substack{\bx+\bv\in\Omega \\ \|\bv\| \leq 1}} 
    \grad f(\bx)^T \bv \right|, \qquad \forall \bx\in\Omega. 
    \label{eq_convex_criticality}
\end{align}
Indeed, we always have $\pi(\bx) \geq 0$, and $\pi(\bx)=0$ if and only 
if $\bx$ is first-order critical 
for~\eqref{eq_convex_cons_problem}~\cite[Theorem 12.1.6]{Conn2000}. Moreover, 
since $\Omega$ is assumed to be closed, the 
problem~\eqref{eq_convex_criticality} always has a solution. Letting 
$\bv^*(\bx)$ denote such a solution, it follows that 
\begin{align}
    \pi(\bx) = -\grad f(\bx)^T \bv^*(\bx), 
	\quad
	\bx+\bv^*(\bx)\in\Omega,
	\quad
	\|\bv^*(\bx)\| \leq 1.
    \label{eq_pi_inner_product_defn}
\end{align}

Deriving convergence guarantees for Algorithm~\ref{alg_ds_basic} requires 
to relate the polling sets to the stationarity 
measure~\eqref{eq_convex_criticality}. To this end, one must rely on more 
than positive spanning sets, and use directions that conform to the 
geometry of nearby constraints.
More precisely, given a feasible point $\bx\in\Omega$ and tolerance 
$\alpha \geq 0$, we define the \emph{nearly active} constraints at $\bx$ to be 
those whose boundaries are within distance $\alpha$ of $\bx$, i.e.
\begin{align}
    I(\bx,\alpha) := \{ i\in\{1,\ldots,m\} : 
    b_i - \alpha\|\ba_i\|^2 \leq \ba_i^T \bx \}, 
    \label{eq_lincons_nearly_active}
\end{align}
noting that we always have $\ba_i^T \bx \leq b_i$ for all $i$ since $\bx$ 
is feasible. The approximate normal cone of $\Omega$ at $\bx$ is then defined 
as
\begin{align}
    N_{\Omega}(\bx,\alpha) := 
    \operatorname{cone}(\{\ba_i : i\in I(\bx,\alpha)\}), 
    \label{eq_lincons_normal_cone}
\end{align}
while the approximate tangent cone is the polar of $N_{\Omega}(\bx,\alpha)$, 
i.e.~$T_{\Omega}(\bx,\alpha) := N_{\Omega}(\bx,\alpha)^{\circ}$. When 
$\alpha=0$, these cones coincide with the usual tangent and normal cones from 
constrained optimization~\cite[Section 12.2]{Nocedal2006}.

To guarantee convergence of Algorithm~\ref{alg_ds_basic}, it suffices to 
use polling sets that contain generators of tangent cones at every iteration.  
Complexity results are obtained assuming sufficient quality of those sets, in 
the sense of an extension of the cosine measure to the linearly constrained 
case~\cite{Gratton2019,Kolda2007}.

\begin{assumption} \label{ass_poll_generators_lincons}
    At every iteration of Algorithm~\ref{alg_ds_basic}, the set $\calD_k$ 
    consists of unit vectors, and satisfies 
    \begin{align*}
        \operatorname{cm}_{T_{\Omega}(\bx_k,\alpha_k)}(\mathcal{D}_k) 
        := 
        \inf_{\substack{\bv\in\R^n \\ 
        	\proj_{T_{\Omega}(\bx_k,\alpha_k)}(\bv) \neq \bm{0}}} 
        \: 
        \max_{\bd\in\mathcal{D}_k} 
        \frac{\bd^T \bv}{\|\bd\| \: \|\proj_{T_{\Omega}(\bx_k,\alpha_k)}(\bv)\|} 
        \ge \kappa 			
    \end{align*}
    for some $\kappa>0$.
\end{assumption}

Note that the use of unit vectors was made for simplicity in Gratton et 
al.~\cite{Gratton2019}, and that the theory can on principle be extended to 
allow for uniform bounds on the direction norms, as in 
Assumption~\ref{ass_kappa_descent}. A key contribution of this paper is to 
allow for the minimum direction norm to vary with the iteration (and the 
stepsize).

Note also that 
$\operatorname{cm}_{T_{\Omega}(\bx_k,\alpha_k)}(\calD_k) \geq \kappa > 0$ 
implies that $\calD_k$ contains a set of generators for the approximate 
tangent cone. Since the only finitely generated cones are 
polyhedral~\cite[Theorem 2.8.8]{Stoer1970}, this strategy cannot be extended to 
general convex sets using polling sets of finite cardinality. In the linearly 
constrained setting, however, the reasoning from the unconstrained case leads to 
both convergence and complexity results~\cite{Gratton2019,Kolda2007}. 

\begin{theorem} \label{thm_lincons_ds_conv_old}
	Let Assumptions~\ref{ass_smoothness}, \ref{ass_omega} 
	and~\ref{ass_poll_generators_lincons} hold. Suppose further that the 
	gradient of $f$ is bounded on $\Omega$. Then, Algorithm~\ref{alg_ds_basic}
    achieves $\pi(\bx_k) \leq \epsilon$ after at most 
    $\bigO(\kappa^{-2} \epsilon^{-2})$ iterations, or 
    $\bigO(p\kappa^{-2} \epsilon^{-2})$ objective evaluations.
    Hence $\liminf_{k\to\infty} \pi (\bx_k) = 0$.
\end{theorem}

Quantifying the value of $p$ and $\kappa$ for arbitrary linear constraints is 
challenging in general. However, when the set $\Omega$ consists only of bound 
constraints, both approximate normal and tangent cones are generated by 
coordinate vectors and their negatives. Letting $\calD_k$ be the subset of 
$\{\pm \be_i\}$ that are feasible for the stepsize $\alpha_k$ guarantees that
Assumption~\ref{ass_poll_generators_lincons} holds with $\kappa=1/\sqrt{n}$, 
as in the unconstrained setting. One then obtains complexity bounds of 
$\bigO(n \epsilon^{-2})$ iterations and 
$\bigO(n^2 \epsilon^{-2})$ objective evaluations, thus matching the 
unconstrained case~\cite{Gratton2019}. We will revisit the bound-constrained 
setting in Section~\ref{sec_construct_bounds}.

\section{$\bm{\Lambda}$-Positive Spanning Sets} 
\label{sec_lbda_pss}

In this section, we introduce an alternative to the cosine measure for 
assessing the quality of a polling set, inspired by model-based 
derivative-free optimization. We illustrate this concept for both unconstrained 
and bound-constrained optimization, deferring the general case to 
Section~\ref{sec_poll_sets}.

\subsection{Motivation: Linear Interpolation} 
\label{sec_motivation}

In model-based optimization, local 
polynomial approximations to the objective $f$ are constructed by 
interpolation to known function values, and replace Taylor-like approximations 
inside trust-region methods \cite{Roberts2025}.
The simplest formulation consists in building a linear interpolant for $f$: 
given a base point $\bx$ and points $\by_1,\ldots,\by_n$ near $\bx$, we find a 
linear (model) function $m:\R^n\to\R$ such that
\begin{align}
    f(\by_i) \approx m(\by_i) := f(\bx) + \bg^T (\by_i-\bx), 
    \qquad i=1,\dots,n,
    \label{eq_linear_model}
\end{align}
by solving the $n\times n$ linear system
\begin{align}
    \begin{bmatrix} 
    	\bd_1^T \\ \vdots \\ \bd_n^T 
    \end{bmatrix} 
    \bg 
    = 
    \begin{bmatrix} 
    	f(\by_1) - f(\bx) \\ \vdots \\ f(\by_n) - f(\bx) 
    \end{bmatrix}, 
    \label{eq_linear_interp}
\end{align}
where $\bd_i := \by_i - \bx$ for $i=1,\ldots,n$.
The vector $\bg$ is also known as a simplex gradient for $f$, and can be used 
to define algorithmic steps~\cite{Regis2015}.

Although the matrix in \eqref{eq_linear_interp} is invertible if and 
only if the vectors $\bd_1,\ldots,\bd_n$ span $\R^n$ in the sense of 
Definition~\ref{def_pss}, the quality of the model is not defined 
according to the cosine measure of this set in model-based DFO. Rather, the 
quality of the model~\eqref{eq_linear_model} relates to how well it can 
approximate the associated Taylor series for $f$ \cite[Lemma 5.1]{Roberts2025}. 
Under Assumption~\ref{ass_smoothness}, for any interpolation point $\by_i$,
we have~\cite[Lemma 5.2]{Roberts2025}:
\begin{align*}
    |(\bg-\grad f(\bx))^T (\by_i-\bx)| 
    &= |m(\by_i) - f(\bx) - \grad f(\bx)^T (\by_i-\bx)|, \\
    &= |f(\by_i) - f(\bx) - \grad f(\bx)^T (\by_i-\bx)|, \\
    &\leq \frac{L}{2}\|\by_i-\bx\|^2 = \frac{L}{2} \|\bd_i\|^2.
\end{align*}
Now consider another point $\by$ near $\bx$.
If the vectors $\bd_1,\ldots,\bd_n$ span $\R^n$, then there exist constants 
$c_1(\by),\ldots,c_n(\by)\in\R$ such that
\begin{align*}
    \by-\bx = \sum_{i=1}^{n} c_i(\by) \bd_i,
\end{align*}
and so we may compute
\begin{align*}
    |m(\by) - f(\bx) - \grad f(\bx)^T (\by-\bx)| = |(\bg - \grad f(\bx))^T (\by-\bx)|
    \leq \frac{L}{2} \left(\max_i \|\bd_i\|^2\right) \sum_{i=1}^{n} |c_i(\by)|.
\end{align*}
In model-based DFO, the points $\by_i$s are usually close to $\bx$, typically within 
a trust region as a ball centered at $\bx$~\cite{Conn2009}. The quality of the 
interpolation set is then measured by the quantity $\sum_{i=1}^{n} |c_i(\by)|$, where 
the values $c_i(\by)$ correspond to the Lagrange polynomials for the interpolation 
points, evaluated at $\by$. An upper bound $\Lambda>0$ on the quantity 
$\sum_{i=1}^{n} |c_i(\by)|$ is called a $\Lambda$-poisedness constant for the 
interpolation set. In the next section, we adapt this concept to positive spanning sets.

\subsection{$\bm{\Lambda}$-Positive Spanning Sets}
\label{ssec_lbdapss}

We now propose an alternate characterization of a positive spanning set, inspired by 
the results of Section~\ref{sec_motivation}. Unlike the classical definition of a PSS, 
Definition~\ref{def_lambda_pss} below assesses the quality of a set within a ball.

\begin{definition} \label{def_lambda_pss}
    Let $\bx \in \R^n$, $\alpha>0$ and $\Lambda > 0$.
    A set $\{\bd_1,\ldots,\bd_p\}\subset\R^n$ is a $\Lambda$-positive spanning set 
    ($\Lambda$-PSS) for $B(\bx,\alpha)$ if, for any $\bv\in B(\bm{0},\alpha)$, there 
    exists $\bc(\bv)\in\R^p$ with $\bc(\bv)\geq \bm{0}$ such that 
    \begin{align}
    \label{eq_lambda_pss}
        \bv=\sum_{i=1}^{p} c_i(\bv) \bd_i
        \quad \mbox{and} \quad
        \|\bc(\bv)\|_1 = \sum_{i=1}^p c_i(\bv) \leq \Lambda.
    \end{align}
\end{definition}

Definition~\ref{def_lambda_pss} is a stronger requirement than the positive spanning 
set property of Definition~\ref{def_pss}. As shown below, bounding the coefficients 
$c_i(\bv)$ tightens the notion of a PSS in the same way than bounding the cosine 
measure from below.

\begin{lemma} \label{lem_pss_implies_descent_new}
    Let $\bx\in\R^n$, $d_{\max}>0$ and $\alpha>0$.
    If $\mathcal{D}:=\{\bd_1,\ldots,\bd_p\}\subset\R^n$ is a $\Lambda$-PSS for 
    $B(\bx,\alpha)$ and $\|\bd_i\| \leq d_{\max} \alpha$ for all $i=1\ldots,p$, 
    then $\operatorname{cm}(\mathcal{D}) \geq \kappa$ with 
    $\kappa := \frac{1}{d_{\max} \Lambda}$.
\end{lemma}
\begin{proof}
    To find a contradiction, suppose that $\operatorname{cm}(\mathcal{D})<\kappa$.
    Then, there exists $\bv\neq\bm{0}$ such that 
    $\bv^T \bd_i < \kappa \|\bv\|\cdot \|\bd_i\|$ for all $i=1,\ldots,p$.
    Now consider $\hat{\bv} := \frac{\alpha}{\|\bv\|}\bv$, so that 
    $\|\hat{\bv}\|=\alpha$. Since $\mathcal{D}$ is a 
    $\Lambda$-PSS for $B(\bx,\alpha)$, we can write 
    $\hat{\bv}=\sum_{i=1}^{p} c_i \bd_i$ for some $c_i \geq 0$ and 
    $\sum_{i=1}^{p} c_i \leq \Lambda$ (we omit the dependency on $\hat{\bv}$ for 
    simplicity). Thus,
    \begin{align}
        \alpha^2 
        = \hat{\bv}^T \hat{\bv} 
        = \sum_{i=1}^{p} c_i \bd_i^T \hat{\bv} 
        = \frac{\alpha}{\|\bv\|} \sum_{i=1}^{p} c_i \bd_i^T \bv 
        < \frac{\alpha}{\|\bv\|} \sum_{i=1}^{p} c_i \kappa \|\bv\| \: \|\bd_i\| 
        \leq \alpha^2 \kappa \Lambda d_{\max} = \alpha^2,
    \end{align}
    where the last equality follows by definition of $\kappa$, and we have a contradiction. 
\end{proof}

\begin{lemma} \label{lem_descent_implies_pss_new}
    Let $\bx\in\R^n$ and $\alpha>0$.
    If $\calD:=\{\bd_1,\ldots,\bd_p\}\subset\R^n$ is a PSS such that 
    $\cm(\mathcal{D}) \geq \kappa>0$ and $\|\bd_i\| \geq d_{\min} \alpha$ for all 
    $i=1,\ldots,p$, then $\mathcal{D}$ is a $\Lambda$-PSS for $B(\bx,\alpha)$ with 
    $\Lambda:=\frac{1}{d_{\min} \kappa}$.
\end{lemma}
\begin{proof}
    To find a contradiction, suppose that $\mathcal{D}$ is not a $\Lambda$-PSS for 
    $B(\bx,\alpha)$. Then, there exists a nonzero $\bv \in B(\bm{0},\alpha)$ such that 
    any decomposition of the form $\bv = \sum_{i=1}^p c_i \bd_i$ with $c_i \ge 0$ 
    requires $\sum_{i=1}^{p} c_i > \Lambda$. Letting $D\in\R^{n\times p}$ have 
    $i$-th column $\bd_i$ and $\be$ denote the vector of all ones, it follows that 
    the linear system
    \begin{align}
    \label{eq_farkassystem}
        D\bc = \bv, \qquad \bc^T \be + b = \Lambda, 
        \qquad \text{and} \qquad 
        \bc \geq \bm{0}, \: b \geq 0,
    \end{align}
    has no solution. By Farkas' lemma (e.g.~\cite[Lemma 12.4]{Nocedal2006}), 
    if~\eqref{eq_farkassystem} has no solution, there must exist 
    $\by_1\in\R^n$ and $y_2\in\R$ such that 
    \begin{align}
    \label{eq_farkassystem2}
        D^T \by_1 + y_2 \be \geq \bm{0}, \qquad y_2 \geq 0, 
        \qquad \text{and} \qquad 
        \bv^T \by_1 + \Lambda y_2 < 0.
    \end{align}
    Since $\Lambda >0$, the last two inequalities in~\eqref{eq_farkassystem2} 
    imply that $\by_1\neq\bm{0}$.
    Hence, because $\operatorname{cm}(\mathcal{D}) \geq \kappa$, there is at 
    least one $i$ such that $\bd_i^T (-\by_1) \geq \kappa \|\by_1\| \|\bd_i\|$. 
    Using that the $i$-th row of $D^T \by_1 + y_2 \be$ is 
    $\bd_i^T \by_1 + y_2$, we obtain that
    $y_2 \geq \kappa\|\by_1\| \|\bd_i\| \geq \kappa d_{\min} \alpha \|\by_1\|$, 
    and thus
    \begin{align}
    \label{eq_contradpsslbda}
        \bv^T \by_1 
        < -\Lambda y_2 
        \le -\Lambda\kappa d_{\min} \alpha \|\by_1\| = -\alpha \|\by_1\|.
    \end{align}
    On the other hand, the Cauchy-Schwarz inequality gives 
    $\bv^T \by_1 \geq -\|\bv\| \: \|\by_1\| \geq -\alpha \|\by_1\|$, which 
    contradicts~\eqref{eq_contradpsslbda}. 
\end{proof}

Although the notion of $\Lambda$-PSS can be identified with requirements on the 
cosine measure, a key difference is that cosine measure is scale-invariant, 
whereas the constant $\Lambda$ depends on the built-in scale $\alpha$. When 
directions are scaled to length $\alpha$, being a $\Lambda$-PSS corresponds to 
having a cosine measure of at least $\kappa=\Lambda^{-1}$ (with 
Definition~\ref{def_pss} corresponding to the limit case 
$\Lambda \rightarrow \infty$). The following example illustrates this property.

\begin{example} \label{ex_pss_uncons}
    Given any $\alpha>0$, the set 
    $\mathcal{D}=\{\pm \alpha\be_1, \ldots, \pm \alpha \be_n\}$ is a PSS
    with $\cm(\calD)=\tfrac{1}{\sqrt{n}}$. Moreover, by
    Lemma~\ref{lem_descent_implies_pss_new}, for any $\bx \in \R^n$, $\calD$ is 
    a $\sqrt{n}$-PSS for $B(\bx,\alpha)$.
\end{example}

The interest of considering $\Lambda$-PSSs lies in allowing directions with 
arbitrarily small norm without jeopardizing theoretical guarantees of 
direct search. In the next section, we will show that complexity guarantees 
can be derived using $\Lambda$-PSSs in lieu of standard PSSs with bounded 
cosine measure.

\section{Complexity of direct search using $\Lambda$-PSSs}
\label{sec_newwcc}

In this section, we revisit the worst-case complexity of direct search by 
assuming that polling directions form $\Lambda$-PSSs. We first analyze the 
unconstrained case (Section~\ref{sec_newwcc_unc}), then move to the main 
focus of this paper, i.e.~the polyhedral convex case 
(Section~\ref{sec_newwcc_cvx}).

\subsection{Unconstrained case}
\label{sec_newwcc_unc}

In order to analyze direct search based on $\Lambda$-PSSs, we adapt the 
basic framework of Algorithm~\ref{alg_ds_basic} so that polling directions are 
defined using the current stepsize. The resulting method is given in 
Algorithm~\ref{alg_ds_basic_new}, and merely differs from 
Algorithm~\ref{alg_ds_basic} in the implicit use of $\alpha_k$ to scale 
the directions upon computation.

\begin{algorithm}[tb]
\begin{algorithmic}[1]
    \Statex \textbf{Inputs:} $\bx_0 \in \Omega$, $\alpha_{\max} > 0$, 
    $\alpha_0 \in (0,\alpha_{\max}]$, $\cdec>0$, 
    $0 < \gammadec < 1 < \gammainc$.
    \vspace{0.2em}
    \For{$k=0,1,2,\ldots$}
        \State Compute a finite polling set $\mathcal{D}_k \subset \R^n$.
        \State If there exists $\bd_k \in \mathcal{D}_k$ such that 
        $\bx_k+\bd_k \in \Omega$ and
        \begin{align}
            f(\bx_k+ \bd_k) < f(\bx_k) - \frac{\cdec}{2} \alpha_k^2, 
            \label{eq_sufficient_decrease_new}
        \end{align}
        set $\bx_{k+1}:=\bx_k+ \bd_k$ and 
        $\alpha_{k+1}:=\min\{\gammainc \alpha_k, \alpha_{\max}\}$ 
        (``successful''). 
        \State Otherwise, set $\bx_{k+1}:=\bx_k$ and 
        $\alpha_{k+1}:=\gammadec \alpha_k$ (``unsuccessful'').
    \EndFor
\end{algorithmic}
\caption{Alternative direct-search method for \eqref{eq_convex_cons_problem}.}
\label{alg_ds_basic_new}
\end{algorithm}

Our analysis will rely on the following key assumption on the polling 
sets. We again emphasize that the directions in $\calD_k$ are scaled 
according to $\alpha_k$, unlike in classical direct-search schemes.

\begin{assumption} \label{ass_pss_new}
    At each iteration, the polling set $\calD_k$ is a $\Lambda$-PSS for 
    $B(\bx_k,\alpha_k)$, and $\|\bd\| \leq d_{\max} \alpha_k$ for all 
    $\bd\in\mathcal{D}_k$, where $\Lambda>0$ and $d_{\max}>0$.
\end{assumption}

Under Assumption~\ref{ass_pss_new}, one can relate the stepsize with 
the gradient norm, akin to classical direct-search theory.

\begin{lemma} \label{lem_eventual_success_new}
    Suppose that we run Algorithm~\ref{alg_ds_basic_new} to 
    solve~\eqref{eq_convex_cons_problem} with $\Omega=\R^n$ under 
    Assumptions~\ref{ass_smoothness} and \ref{ass_pss_new}.
    If $\|\grad f(\bx_k)\| \neq \bm{0}$ and 
    $\alpha_k < \frac{2\|\grad f(\bx_k)\|}{(L d_{\max}^2 + \cdec) \Lambda}$, 
    then iteration $k$ is successful.
\end{lemma}
\begin{proof}
    To find a contradiction, suppose iteration $k$ is unsuccessful.
    By Assumption~\ref{ass_smoothness}, it follows that
    \begin{align*}
        -\frac{\cdec}{2}\alpha_k^2 
        \leq f(\bx_k+\bd) - f(\bx_k) 
        \leq \bd^T \grad f(\bx_k) + \frac{L}{2} \|\bd\|^2 
        \leq \bd^T \grad f(\bx_k) + \frac{L}{2} d_{\max}^2 \alpha_k^2,
    \end{align*}
    for all $\bd\in\mathcal{D}_k$.
    Since $\alpha_k>0$, this rearranges to
    \begin{align}
        -\frac{L d_{\max}^2+\cdec}{2} \: \alpha_k^2 \leq \bd^T \grad f(\bx_k), 
        \label{eq_unsuccess_condition_new}
    \end{align}
    for all $\bd\in\mathcal{D}_k$.
    
    Define now $\bv_k := -\frac{\alpha_k}{\|\grad f(\bx_k)\|} \grad f(\bx_k)$ so 
    that $\|\bv_k\| \leq \alpha_k$ and 
    $\alpha_k \|\grad f(\bx_k)\| = -\bv_k^T \grad f(\bx_k)$.
    Since $\mathcal{D}_k=\{\bd_1,\ldots,\bd_p\}$ is a $\Lambda$-PSS for 
    $B(\bx_k,\alpha_k)$, we may write $\bv_k = \sum_{i=1}^{p} c_i \bd_i$ for some 
    constants $c_i\geq 0$ with $\sum_{i=1}^{p} c_i \leq \Lambda$.
    Using \eqref{eq_unsuccess_condition_new}, we then obtain
    \begin{align*}
        \|\grad f(\bx_k)\| = -\frac{1}{\alpha_k}\bv_k^T \grad f(\bx_k) 
        = -\frac{1}{\alpha_k}\sum_{i=1}^{p} c_i \bd_i^T \grad f(\bx_k), 
        &\leq \frac{L d_{\max}^2+\cdec}{2} \: \alpha_k \sum_{i=1}^{p} c_i, \\
        &\leq \frac{L d_{\max}^2+\cdec}{2}\: \alpha_k  \Lambda,
    \end{align*}
    which contradicts our assumption that 
    $\alpha_k < \frac{2\|\grad f(\bx_k)\|}{(L d_{\max}^2 + \cdec) \Lambda}$.
\end{proof}

Equipped with Lemma~\ref{lem_eventual_success_new}, we can derive a worst-case 
complexity bound for Algorithm~\ref{alg_ds_basic_new}. Although the proof is 
identical to that of a direct-search method based on PSSs, we provide it 
below for sake of completeness.

\begin{theorem} \label{thm_wcc_new}
    Suppose that we run Algorithm~\ref{alg_ds_basic_new} to 
    solve~\eqref{eq_convex_cons_problem} with $\Omega=\R^n$ under 
    Assumptions~\ref{ass_smoothness} and \ref{ass_pss_new}.
    Then, for any $\epsilon>0$, 
    $\|\grad f(\bx_k)\|\leq \epsilon$ occurs for the first time after 
    at most
    \begin{align}
    \label{eq_wcc_new}
        \left(1 + \frac{\log(\gammainc)}{\log(\gammadec^{-1})}\right) 
        \left[\frac{2[f(\bx_0)-\flow]}{\cdec \alpha_{\min}^2}\right] 
        + \frac{\log(\alpha_0/\alpha_{\min})}{\log(\gammadec^{-1})}
    \end{align}
    iterations, where
    \begin{align}
        \alpha_{\min} 
        := 
        \gammadec \frac{2\epsilon}{(L d_{\max}^2+\cdec)\Lambda}. 
        \label{eq_alpha_min_new}
    \end{align}
\end{theorem}
\begin{proof}
    Suppose that $\min_{k=0,\ldots,K-1} \|\grad f(\bx_k)\| > \epsilon$.
    Lemma~\ref{lem_eventual_success_new} then implies that any iteration 
    $k$ such that $\alpha_k < \frac{2\epsilon}{(L d_{\max}^2+\cdec)\Lambda}$ 
    is successful, and results in an increase of $\alpha_k$. Combining this 
    observation with the update rules on $\alpha_k$, we find that  
    $\alpha_k \ge \alpha_{\min}$ for every $k=0,\dots,K-1$, where 
    $\alpha_{\min}$ is given by~\eqref{eq_alpha_min_new}.
    
    Let $\mathcal{S}=\{k\leq K-1 : \text{iteration $k$ is successful}\}$ and 
    $\mathcal{U} = \{k\leq K-1 : \text{iteration $k$ is unsuccessful}\}$. 
    On one hand, summing \eqref{eq_sufficient_decrease_new} over 
    $k\in\mathcal{S}$ gives
    \begin{align*}
        f(\bx_0) - \flow \geq \sum_{k\in\mathcal{S}} f(\bx_k) - f(\bx_{k+1}) 
        \geq \frac{\cdec}{2}\alpha_{\min}^2 |\mathcal{S}|,
    \end{align*}
    hence
    \begin{align}
    \label{eq_bounditsucc_unc}
        |\mathcal{S}| \leq \frac{2[f(\bx_0)-\flow]}{\cdec\alpha_{\min}^2}.
    \end{align}
    On the other hand, the updating rules on $\alpha_k$ give
    \begin{align*}
        \alpha_{\min} 
        \leq \alpha_k 
        \leq \alpha_0 \gammainc^{|\mathcal{S}|} \gammadec^{|\mathcal{U}|},
    \end{align*}
    from which we bound the number of unsuccessful iterations as follows:
    \begin{align}
    \label{eq_bounditunsucc_unc}
        |\mathcal{U}| 
        \leq 
        \frac{1}{\log(\gammadec^{-1})} 
        \left[\log\left(\frac{\alpha_0}{\alpha_{\min}}\right) 
        + |\mathcal{S}| \log \gammainc\right].
    \end{align}
    Putting~\eqref{eq_bounditsucc_unc} and~\eqref{eq_bounditunsucc_unc} 
    together, we arrive at
    \begin{align*}
        K = |\mathcal{S}| + |\mathcal{U}| 
        &\leq 
        \left(1 + \frac{\log(\gammainc)}{\log(\gammadec^{-1})}\right) |\mathcal{S}| 
        + \frac{\log(\alpha_0/\alpha_{\min})}{\log(\gammadec^{-1})}, \\
        &\leq 
        \left(1 + \frac{\log(\gammainc)}{\log(\gammadec^{-1})}\right) 
        \left[\frac{2[f(\bx_0)-\flow]}{\gamma\alpha_{\min}^2}\right] 
        + \frac{\log(\alpha_0/\alpha_{\min})}{\log(\gammadec^{-1})},
    \end{align*}
    and we get the desired result.
\end{proof}

The complexity bound~\eqref{eq_wcc_new} is $\bigO(\Lambda^2 \epsilon^{-2})$, 
and implies that a $\bigO(p \Lambda^2 \epsilon^{-2})$ bound on objective 
evaluations holds assuming polling sets have at most $p$ vectors. Using 
$\kappa\sim \frac{1}{\Lambda}$ from Lemma~\ref{lem_pss_implies_descent_new}, 
we recover the standard complexity bounds of $\bigO(\kappa^{-2}\epsilon^{-2})$ 
iterations and $\bigO(p\kappa^{-2}\epsilon^{-2})$ evaluations stated in 
Theorem~\ref{thm_ds_wcc_old}.

\subsection{Polyhedral convex-constrained case}
\label{sec_newwcc_cvx}

In this section, we return to the generic case of a constrained set $\Omega$ 
satisfying Assumption~\ref{ass_omega}. Note that the analysis below applies to 
any convex constraint set with nonempty interior, even though the algorithm is 
only implementable using finite polling sets when $\Omega$ is polyhedral. 

To establish complexity guarantees for Algorithm~\ref{alg_ds_basic_new} in 
this setting, we introduce the following constrained version of the 
$\Lambda$-PSS property.

\begin{definition} \label{def_lambda_pss_convex_new}
    Consider a point $\bx\in\Omega$ and radius $\alpha>0$, and a constant 
    $\Lambda\geq 0$. A polling set $\{\bd_1,\ldots,\bd_p\}\subset\R^n$ is a 
    \emph{$\Lambda$-positive spanning set} ($\Lambda$-PSS) for 
    $B(\bx,\alpha)\cap \Omega$ if $\bx+\bd_i\in\Omega$ for all $i=1,\ldots,p$ 
    and, for any $\bv\in\R^n$ with $\bx+\bv\in\Omega$ and 
    $\|\bv\|\leq \alpha$, there exists $\bc(\bv)\in\R^p$ with 
    $c_i(\bv)\geq 0$ such that $\bv=\sum_{i=1}^{p} c_i(\bv) \bd_i$, and 
    $\|\bc(\bv)\|_1 \leq \Lambda$.
\end{definition}


Unlike Definition~\ref{def_lambda_pss} above, 
Definition~\ref{def_lambda_pss_convex_new} makes explicit use of the point 
$\bx$ to guarantee feasibility of all directions, as well as to restrict 
the $\Lambda$-PSS property to feasible vectors $\bv$ within $B(\bm{0},\alpha)$. 
We will thus analyze Algorithm~\ref{alg_ds_basic_new} under the following 
assumption.

\begin{assumption} \label{ass_pss_convex_new}
    At each iteration of Algorithm~\ref{alg_ds_basic_new}, the polling set 
    $\mathcal{D}_k$ is a $\Lambda$-PSS 
    for $B(\bx_k,\alpha_k)\cap \Omega$, and $\|\bd\| \leq d_{\max} \alpha_k$ 
    for all $\bd\in\mathcal{D}_k$, where $\Lambda>0$ and $d_{\max}>0$.
\end{assumption}

Under Assumption~\ref{ass_pss_convex_new}, Algorithm~\ref{alg_ds_basic_new} 
produces only feasible iterates and trial points. As a result, our 
analysis will focus on bounding the number of iterations and evaluations 
necessary to achieve $\pi(\bx_k) \leq \epsilon$, as in 
Section~\ref{sec_background_lincons}. We follow the same reasoning than 
in the unconstrained setting.

\begin{lemma} \label{lem_eventual_success_convex_new}
	Consider the $k$th iteration of Algorithm~\ref{alg_ds_basic_new} 
	under Assumptions~\ref{ass_smoothness}, \ref{ass_omega} 
    and~\ref{ass_pss_convex_new}.
    If $\pi(\bx_k) \neq 0$ and 
    $
        \alpha_k 
        < 
        \min\left(\frac{2 \pi(\bx_k)}{(L d_{\max}^2 + \cdec) \Lambda}, 1\right),
    $
    then iteration $k$ is successful.
\end{lemma}
\begin{proof}
    To find a contradiction, suppose iteration $k$ is unsuccessful. 
    By the same reasoning as in the unconstrained case 
    (Lemma~\ref{lem_eventual_success_new}), it follows that 
    \begin{align}
        -\frac{L d_{\max}^2+\cdec}{2} \: \alpha_k^2 \leq \bd^T \grad f(\bx_k), \label{eq_unsuccess_condition_convex_new}
    \end{align}
    for all $\bd\in\mathcal{D}_k$.

    Defining $\bv_k := \bv^*(\bx_k)$ as in \eqref{eq_pi_inner_product_defn}, we 
    have $\pi(\bx_k) = -\grad f(\bx_k)^T \bv_k$ with $\bx_k+\bv_k\in\Omega$ and 
    $\|\bv_k\| \leq 1$. Moreover, since $\pi(\bx_k) \neq 0$ by assumption, we 
    also have $\bv_k\neq\bm{0}$.

    Let $\hat{\bv}_k := \alpha_k \bv_k$, so $\|\hat{\bv}_k\| \leq \alpha_k$ and 
    $\bx_k+\hat{\bv}_k\in \Omega$ by convexity of $\Omega$ (recall that 
    $\alpha_k<1$). From Assumption~\ref{ass_pss_convex_new}, we 
    may write $\hat{\bv}_k = \sum_{i=1}^{p} c_i \bd_i$ for constants $c_i \geq 0$ 
    with $\sum_{i=1}^{p} c_i \leq \Lambda$.
    Using \eqref{eq_unsuccess_condition_convex_new}, we get
    \begin{align*}
        \pi(\bx_k) 
        = -\bv_k^T \grad f(\bx_k) 
        = -\frac{1}{\alpha_k} \hat{\bv}_k^T \grad f(\bx_k) 
        = \frac{1}{\alpha_k}\sum_{i=1}^{p} c_i \bd_i^T \grad f(\bx_k) 
        \leq \frac{L d_{\max}^2 + \cdec}{2} \: \alpha_k \Lambda,
    \end{align*}
    which contradicts $\alpha_k< \frac{2\pi(\bx_k)}{(Ld_{\max}^2+\sigma)\Lambda}$.
\end{proof}

\begin{theorem} \label{thm_wcc_convex_new}
    Suppose that we run Algorithm~\ref{alg_ds_basic_new} to 
    solve~\eqref{eq_convex_cons_problem} under 
    Assumptions~\ref{ass_smoothness}, \ref{ass_omega} 
    and~\ref{ass_pss_convex_new}. 
    Then, for any $\epsilon>0$, $\pi(\bx_k)\leq \epsilon$ occurs for the 
    first time after at most
    \begin{align}
        \left(1 + \frac{\log(\gammainc)}{\log(\gammadec^{-1})}\right) 
        \left[\frac{2[f(\bx_0)-\flow]}{\gamma \hat{\alpha}_{\min}^2}\right] 
        + \frac{\log(\alpha_0/\hat{\alpha}_{\min})}{\log(\gammadec^{-1})}
        \label{eq_wcc_convex_new}
    \end{align}
    iterations, where
    \begin{align}
        \hat{\alpha}_{\min} 
        := 
        \gammadec \min\left(\frac{2\epsilon}{(L d_{\max}^2+\cdec)\Lambda}, 1\right). 
        \label{eq_alpha_min_convex_new}
    \end{align}
\end{theorem}
\begin{proof}
    Suppose that $\min_{k=0,\ldots,K-1} \pi(\bx_k) > \epsilon$.
    From Lemma~\ref{lem_eventual_success_convex_new}, we know that 
    any iteration $k$ such that 
    $\alpha_k < \min(\frac{2\epsilon}{(L d_{\max}^2+\cdec)\Lambda}, 1)$ 
    is successful, and leads to an increase in $\alpha_k$.
    Hence for all iterations $k=0,\ldots,K-1$, we must have 
    $\alpha_k \geq \hat{\alpha}_{\min}$ with $\hat{\alpha}_{\min}$ defined in \eqref{eq_alpha_min_convex_new}.

    The remainder of the proof is identical to that of 
    Theorem~\ref{thm_wcc_new}.
\end{proof}

As in the unconstrained case, the complexity bound~\eqref{eq_wcc_convex_new} 
simplifies to $\bigO(\Lambda^2 \epsilon^{-2})$ iterations for small $\epsilon$, 
and yields a bound of $\bigO(p \Lambda^2 \epsilon^{-2})$ objective evaluations 
when the polling sets have $p$ vectors. The dependency on $\epsilon$ is the same 
than other DFO and derivative-based methods applied to convex-constrained 
problems~\cite{Cartis2012,Hough2022}. In the next section, we investigate the 
dependency on $n$ through $\Lambda$ and $p$ by explicit construction of 
$\Lambda$-PSS.


\section{Constructing a $\bm{\Lambda}$-PSS with explicit linear constraints} 
\label{sec_poll_sets}

We now describe the generation of $\Lambda$-PSSs when the constraint set 
$\Omega$ is given by~\eqref{eq_linear_cons} and satisfies
Assumption~\ref{ass_omega}. We first consider the case of bound constraints, 
then extend our approach to the general case.

\subsection{Bound Constraints} 
\label{sec_construct_bounds}

Suppose first that the feasible set is given by
\begin{align}
    \Omega = \{ \bx\in\R^n : \bx^L \leq \bx \leq \bx^U\},
    \label{eq_omega_bounds}
\end{align}
where $x^L_i < x^U_i$, $i=1,\ldots,n$. In this setting , given a (feasible) point 
$\bx$ and a stepsize $\alpha>0$, the approximate tangent cone is generated 
by positive and negative coordinate directions $\pm \be_i$ such that 
$\bx\pm\alpha\be_i \in \Omega$, while the approximate normal cone is generated 
by the remaining positive and negative coordinate directions. Using these 
vectors as $\calD$ guarantees that
$\cm_{T_{\Omega}(\bx,\alpha)}(\calD) \ge \tfrac{1}{\sqrt{n}}$~\cite{Gratton2019}.

To define a $\Lambda$-PSS for the same pair ($\bx$,$\alpha$), we scale all 
positive and negative directions to ensure feasibility, i.e. we consider
\begin{align}
    \calD = \cup_{i=1}^{n} \{-\alpha_{-i} \be_i, \alpha_i \be_i\}, 
    \label{eq_pss_bounds}
\end{align}
where the scalings $\alpha_{\pm i}$ are given by
\begin{align}
    \alpha_{-i} := \min(\alpha, x_i-x^L_i) \qquad \text{and} \qquad
    \alpha_i := \min(\alpha, x^U_i-x_i).
    \label{eq_alphai_bounds}
\end{align}
Note that all the scaling values~\eqref{eq_alphai_bounds} lie in 
$[0,\alpha]$ since $\bx \in \Omega$. Importantly, the value $\alpha_{\pm i}=0$ 
is allowed, and amounts to discarding the corresponding direction 
(note that $\alpha_i + \alpha_{-i}>0$ since $x_i^L < x_i^U$). The next 
result quantifies the quality of the set~\eqref{eq_pss_bounds}.

\begin{theorem} \label{thm_pss_bounds}
    Let $\Omega$ be given by~\eqref{eq_omega_bounds}. Given any 
    $\bx \in \Omega$ and $\alpha>0$, the set defined by~\eqref{eq_pss_bounds} 
    is a $\Lambda$-PSS for $B(\bx,\alpha)\cap \Omega$ with 
	\begin{equation}
	\label{eq:thm_pss_bounds}
		\Lambda 
		\; = \;
		\min\left[n,
		\frac{\sqrt{n}\alpha}{\min\left\{\alpha,
		\min_{x_i \neq x_i^L} x_i-x_i^L, 
		\min_{x_i \neq x_i^U} x_i^U-x_i\right\}} 
		\right].
	\end{equation}	    
\end{theorem}
\begin{proof}
	The definition of $\alpha_{\pm i}$ in~\eqref{eq_alphai_bounds} 
	ensures that $\bx+\bd\in \Omega$ for any $\bd \in \calD$.

    Consider now any vector $\bv\in\R^n$ such that $\bx+\bv\in\Omega$ and 
    $\|\bv\|\leq\alpha$. We have 
    $\bv=\sum_{i=1}^{n} v^+_i \be_i + v^-_i (-\be_i)$ where 
    $v^+_i = \max(\bv_i,0)$ and $v^-_i=\max(-v_i,0)$ are the positive and 
    negative parts of $\bv_i$, respectively. We can then write $\bv$ in terms of 
    the polling directions~\eqref{eq_pss_bounds} as 
    \begin{align}
        \bv 
        = \sum_{i=1}^{n} c_i (\alpha_i \be_i) + c_{-i} (-\alpha_{-i} \be_i), 
        \label{eq_poll_bounds_tmp1} 
    \end{align}
    where $c_i = v^+_i / \alpha_i$ when $\alpha_i>0$ and $c_i=0$ otherwise, 
    and $c_{-i} = v^-_i / \alpha_{-i}$ when $\alpha_{-i}>0$ and $c_{-i}=0$ 
    otherwise. To obtain the desired result, we will bound each $c_i+c_{-i}$ 
    separately, noting that the bound can be refined if $v_i=0$ for any $i$, 
    given that this implies $c_i=c_{-i}=0$.
    
    Suppose first that $\alpha_i>0$ and $\alpha_{-i}=0$. In that case, we 
    have $v_i^+ = v_i = |v_i|$ and
    \begin{equation}
    \label{eq:alphaMi0}
    	c_i+c_{-i} = c_i = \frac{v_i^+}{\alpha_i}
    	= \frac{|v_i|}{\min(\alpha,x_i^U-x_i)} 
	\end{equation}     
	where $\alpha_i>0$ guarantees that $x_i<x_i^U$. From $\|v\| \le \alpha$, 
	we get that $|v_i| \le \alpha$, while $\bx+\bv \in \Omega$ ensures that 
	$|v_i|=v_i^+ \le x_i^U - x_i$. It follows that $c_i \le 1$.
	
	Suppose now that $\alpha_{i}=0$ and $\alpha_{-i}>0$. The same argument 
	than above leads to
	\begin{equation}
	\label{eq:alphaPi0}
		c_i+c_{-i} = c_{-i} 
		= \frac{v_i^-}{\alpha_{-i}}
		= \frac{|v_i|}{\min(\alpha,x_i-x_i^L)},
	\end{equation}
	where again the properties of $\bv$ guarantee that $c_{-i} \le 1$.
    
    Finally, suppose that $\alpha_i>0$ and $\alpha_{-i}>0$. In that case,
    \begin{equation}
    \label{eq:alphaPMpos}
    	c_i+c_{-i} 
    	= \frac{v_i^+}{\alpha_i} + \frac{v_i^-}{\alpha_{-i}}
    	= \frac{v_i^+}{\min(\alpha,x_i^U-x_i)} 
    	+ \frac{v_i^-}{\min(\alpha,x_i - x_i^L)} 
    	\le \frac{|v_i|}{\min(\alpha,x_i-x_i^L,x_i^U-x_i)} \le 1.
  	\end{equation}
  	
  	Combining~\eqref{eq:alphaMi0}, \eqref{eq:alphaPi0} and~\eqref{eq:alphaPMpos}, 
  	we arrive at
  	\begin{equation}
  	\label{eq:boundcPicMi}
  		c_i + c_{-i} 
  		\le 
  		\frac{|v_i|}{\min(\alpha,\delta_i)},
  		\qquad 
  		\delta_i:=\left\{
  			\begin{array}{ll}
  				x_i-x_i^L &\mbox{if\ $x_i=x_i^U$}\\
  				x_i^U-x_i &\mbox{if\ $x_i=x_i^L$}\\
  				\min(x_i-x_i^L,x_i^U-x_i) &\mbox{otherwise,}
  			\end{array}
  		\right.
  	\end{equation}
  	and the right-hand side of~\eqref{eq:boundcPicMi} is smaller than $1$.
  	Summing~\eqref{eq:boundcPicMi} over all indices $i$ and using that 
  	$\|\bv\|_1 \le \sqrt{n}\|\bv\| \le \sqrt{n}\alpha$, we obtain
  	\begin{equation}
  	\label{eq:boundsumcPicMi}
  		\sum_{i=1}^n (c_i+c_{-i}) 
  		\le 
  		\sum_{i=1}^n \frac{|v_i|}{\min(\alpha,\delta_i)} 
  		\le \frac{1}{\min(\alpha,\min_i \delta_i)}\|\bv\|_1 
  		\le \frac{\sqrt{n}\alpha}{\min(\alpha,\min_i \delta_i)}.
  	\end{equation}
  	Finally, using that $c_i+c_{-i} \le 1$ for every $i$ ensures that 
  	$\sum_{i=1}^n (c_i+c_{-i}) \le n$, which together 
  	with~\eqref{eq:boundsumcPicMi} gives~\eqref{eq:thm_pss_bounds}.
\end{proof}

The formula~\eqref{eq:thm_pss_bounds} reduces $\Lambda=\sqrt{n}$ when no 
constraints are approximately active for $(\bx,\alpha)$ (or when the 
problem is unconstrained), which matches the cosine measure of tangent 
cone generators in both the unconstrained and bound-constrained setting, 
where $\kappa=\tfrac{1}{\sqrt{n}}$~\cite{Gratton2019}. In the general 
setting, however, we have $\sqrt{n} \le \Lambda \le n$, and the upper 
bound is attained when $n$ linearly independent bound constraints are 
approximately active, which can occur near an extreme point.
This increase in $\Lambda$ seems unavoidable when using directions beyond 
tangent cone generators, but affects our complexity guarantees. Still, our 
polling set choice provides a richer description of the feasible region, 
and can lead to faster convergence, as illustrated by 
Example~\ref{ex_bounds} and our numerical results in 
Section~\ref{sec_numerics}.

\begin{example}
\label{ex_bounds}
    Suppose we have the simple 1D minimization problem
    \begin{align}
        \min_{x\in\R} -x, \quad \text{s.t.} \quad x \leq 1.1,
    \end{align}
    with solution $x^*=1.1$, and consider running 
    Algorithm~\ref{alg_ds_basic_new} with either 
    $\mathcal{D}_k$ taken as generators of the approximate tangent cone 
    (per \cite{Gratton2019}), or as \eqref{eq_pss_bounds}.
    For demonstration purposes, we take $x_0=0$, $\alpha_0=1$, the standard 
    values $\gammadec=0.5$ and $\gammainc=2$, and assume $\gamma<1$ so that 
    all positive steps are successful in the analysis below.

    Using just the approximate tangent cone, our initial polling set is 
    $\mathcal{D}_0=\{-1,1\}$, which gives the successful step $x_1=x_0+1=1$ 
    and $\alpha_1=2$. Then, $\mathcal{D}_1=\{-2\}$ and so the second iteration 
    is unsuccessful.
    We continue having unsuccessful iterations until $\alpha_k \leq x^*-x_k$, 
    and so $\mathcal{D}_k=\{-\alpha_k,\alpha_k\}$ and we get $x_{k+1}=x_k+\alpha_k$.
    Since $\alpha_k$ can only take integer powers of 2 (by choice of $\gammadec$, 
    $\gammainc$), this first happens for $\alpha_6=2^{-4}$ with $x_7=1+2^{-4}=1.0625$.
    After this, the next successful iteration comes from $\alpha_9=2^{-5}$ giving 
    $x_{10}=1+2^{-4}+2^{-5}=1.09375$. Continuing this way, our iterates $x_k$ 
    correspond to (truncated) binary fractions that under-approximate $x^*$. Since 
    $x^*$ has an infinite binary expansion, we converge linearly, but never 
    actually reach $x^*$ in finite time.

    However, if we instead use \eqref{eq_pss_bounds}, we get the same first 
    iteration with $x_1=1$ and $\alpha_1=2$. Then, we have $\calD_1=\{-2, 0.1\}$ 
    so that $x_2=x_1+0.1=x^*$ and we converge to the solution.
\end{example}


%

\subsection{Independent Linear Inequality Constraints}
\label{sec_construct_gen}

We now come back to general polyhedral sets of the form~\eqref{eq_linear_cons} 
satisfying Assumption~\ref{ass_omega}. Existing theory summarized in 
Section~\ref{sec_background_lincons} shows that a polling set is of sufficient 
quality for a pair $(\bx,\alpha) \in \Omega \times \R_{>0}$ provided it 
contains a set of generators for the approximate tangent cone 
$T_{\Omega}(\bx,\alpha)$. The polar decomposition~\cite{Moreau1962} 
\begin{align*}
    \bv 
    = 
    \proj_{T_{\Omega}(\bx,\alpha)}(\bv) + \proj_{N_{\Omega}(\bx,\alpha)}(\bv), 
    \qquad \forall \bv\in\R^n,
\end{align*}
suggests that generators from the approximate normal cone 
$N_{\Omega}(\bx,\alpha)$ could provide additional directions when scaled 
appropriately. This strategy was found beneficial in direct-search 
implementations~\cite{Lewis2007}, even though it produces polling sets for 
which Assumption~\ref{ass_poll_generators_lincons} does not hold (and thus 
theoretical guarantees are lost).

One may naturally wonder whether the sets used in practice form $\Lambda$-PSSs. 
As shown by the example below, this is true but the constant $\Lambda$ can 
be arbitrarily large, while we provide another $\Lambda$-PSS construction 
with bounded $\Lambda$.

\begin{figure}[tb]
  \centering
  \begin{subfigure}[b]{0.48\textwidth}
    \includegraphics[width=\textwidth]{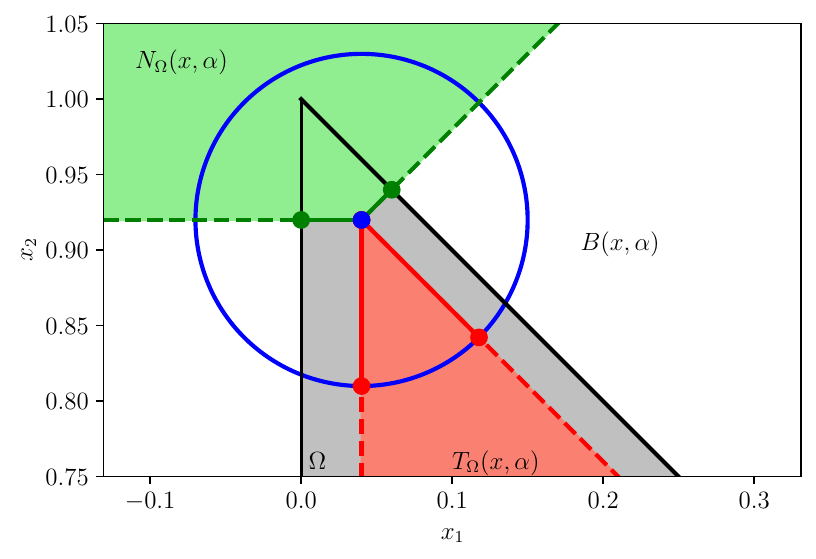}
    \caption{Scaled generators of $N_{\Omega}(\bx,\alpha)$}
    \label{fig_pyramid1}
  \end{subfigure}
  ~
  \begin{subfigure}[b]{0.48\textwidth}
    \includegraphics[width=\textwidth]{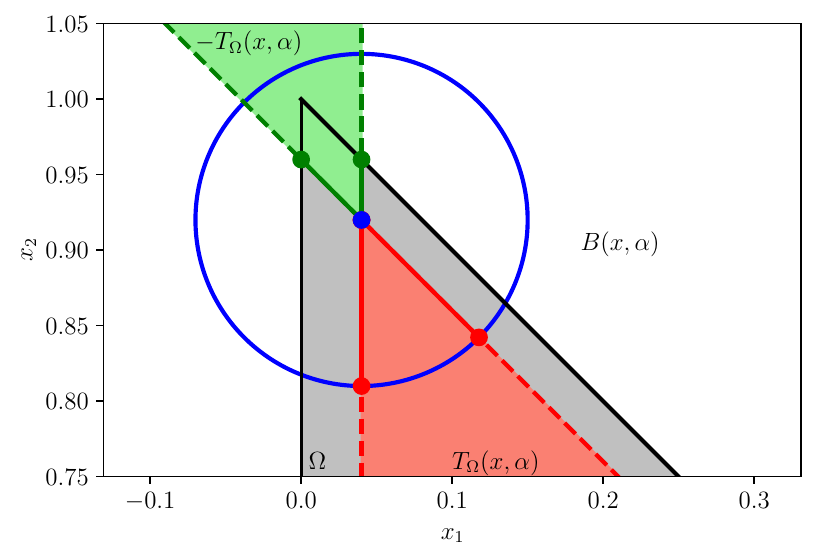}
    \caption{Scaled generators of $-T_{\Omega}(\bx,\alpha)$}
    \label{fig_pyramid2}
  \end{subfigure}
  \caption{Illustration of Example~\ref{ex_pyramid}. A polling 
  set is built by combining generators of $T_{\Omega}(\bx,\alpha)$ 
  with scaled generators of $N_{\Omega}(\bx,\alpha)$ or 
  $-T_{\Omega}(bx,\alpha)$.}
  \label{fig_pyramid}
\end{figure}

\begin{example} 
\label{ex_pyramid}
	Consider $\Omega=\{\bx \in \R^2\ |\ x_1 \ge 0, x_1+x_2 \le 1\}$, 
	and let $\bx=[\epsilon_1, 1-\epsilon_2]^T$ with 
	$\epsilon_1< \epsilon_2\ll 1$. Given a large enough $\alpha>0$ 
	such that the vertex $[0,1]^T \in B(\bx,\alpha)$, both constraints 
	are nearly active at $\bx$. Thus, the generators of 
	$T_{\Omega}(\bx,\alpha)$ are $[0,-1]^T$ and $[1, -1]^T$, while 
	the generators of $N_{\Omega}(\bx,\alpha)$ are $[-1,0]^T$ and 
	$[1, 1]^T$. Figure~\ref{fig_pyramid1} shows a polling set built 
	using those generators, namely
    \begin{align}
        \mathcal{D} = 
        \underbrace{
        \left\{
        \begin{bmatrix} 
    	    0 \\ 
	        -\alpha 
        \end{bmatrix}, 
        \begin{bmatrix} 
        	\alpha/\sqrt{2} \\ 
        	-\alpha/\sqrt{2} 
        \end{bmatrix} 
        \right\}
        }_{T_{\Omega}(\bx,\alpha)} 
        \cup 
        \underbrace{
        \left\{ 
        \begin{bmatrix} -\epsilon_1 \\ 0 \end{bmatrix}, 
        \begin{bmatrix} (\epsilon_2-\epsilon_1)/2 \\ 
        (\epsilon_2-\epsilon_1)/2 \end{bmatrix}
        \right\}
        }_{N_{\Omega}(\bx,\alpha)}.
    \end{align}
    Those vectors form a $\Lambda$-PSS in the sense of 
    Definition~\ref{def_lambda_pss_convex_new}. However, to 
    represent the vector $\bv=[-\epsilon_1, \epsilon_2]^T$ 
    (i.e.~$\bx+\bv=[1,1]^T$) as a positive linear combination of 
    these generators, we must write
    \begin{align}
    \label{eq:v_pyramid1}
        \bv 
        = 
        \frac{\epsilon_1+\epsilon_2}{\epsilon_1} 
        \begin{bmatrix} 
        	-\epsilon_1 \\ 
        	0 
       	\end{bmatrix} 
        + \frac{2\epsilon_2}{\epsilon_2-\epsilon_1} 
        \begin{bmatrix} 
        	(\epsilon_2-\epsilon_1)/2 \\ 
        	(\epsilon_2-\epsilon_1)/2 
        \end{bmatrix}.
    \end{align}
    The positive linear combination coefficients 
    in~\eqref{eq:v_pyramid1} grow arbitrarily large by taking 
    $\epsilon_1$ and $\epsilon_2$ small enough 
    (i.e. making $\bx$ arbitrarily close to $[0 1]^T$). As a 
    result, the value of $\Lambda$ can be made arbitrarily large.
    
    Figure~\ref{fig_pyramid2} shows another polling set obtained 
    by replacing the generators of $N_{\Omega}(\bx,\alpha)$ by 
    that of $-T_{\Omega}(\bx,\alpha)$, with proper scaling to 
    ensure feasibility. Mathematically, this set is
    \begin{align*}
        \mathcal{D}^{'} 
        = 
        \underbrace{
        \left\{
        \begin{bmatrix} 
        	0 \\ 
        	-\alpha 
        \end{bmatrix}, 
        \begin{bmatrix} 
        	\alpha/\sqrt{2} \\ 
        	-\alpha/\sqrt{2} 
        \end{bmatrix} 
        \right\}
        }_{T_{\Omega}(\bx,\alpha)} 
        \cup 
        \underbrace{
        \left\{ 
        \begin{bmatrix} 
        	0 \\ 
        	\epsilon_2-\epsilon_1 
        \end{bmatrix}, 
        \begin{bmatrix} 
        	-\epsilon_1 \\ 
        	\epsilon_1 
        \end{bmatrix} 
        \right\}
        }_{-T_{\Omega}(\bx,\alpha)}.
    \end{align*}
    Now, for the same $\bv$ than in~\eqref{eq:v_pyramid1}, the positive 
    linear combination with columns of $\calD^{'}$ is
    \begin{align}
    \label{eq:v_pyramid2}
        \bv = 
        \begin{bmatrix} 
        	0 \\ 
        	\epsilon_2-\epsilon_1 
        \end{bmatrix} 
        + 
        \begin{bmatrix} 
        	-\epsilon_1 \\ 
        	\epsilon_1 
        \end{bmatrix},
    \end{align}
    whose coefficients are uniformly bounded for all values of 
    $\epsilon_1$ and $\epsilon_2$.
\end{example}

Motivated by Example~\ref{ex_pyramid}, we propose to construct our 
polling set by taking the generators of $T_{\Omega}(\bx,\alpha)$ and their 
negatives, scaled to ensure feasibility. If the generators of 
$T_{\Omega}(\bx,\alpha)$ do not (linearly) span $\R^n$, then we will also need 
to include vectors that positively span the corresponding null space.
We will prove that this construction yields a $\Lambda$-PSS (i.e.~satisfies 
Assumption~\ref{ass_pss_convex_new}) in the case where 
$\{\ba_i : i\in I(\bx,\alpha)\}$ are linearly independent.

First, we state some characterizations of polar cones.

\begin{lemma} \label{lem_cone_technical}
	Let $A=[\ba_1\ \cdots\ \ba_q] \in \R^{n \times q}$ and consider the cones
	\[
		K_1 := \{ \by\in\R^n : \ba_j^T \by \leq 0, \:\: \forall j=1,\ldots,q \}, 
		\quad \text{and} \quad K_2 := \operatorname{cone}(\{\ba_1,\ldots,\ba_q\}).
	\]
	Then, the following properties hold:
    \begin{enumerate}[label=(\alph*)]
        \item $K_1^{\circ} = K_2$ and $K_2^{\circ} = K_1$. \label{lem_cone_1c}
        \item If $q\leq n$ and $\operatorname{rank}(A)=q$, then 
        \label{lem_cone_2a}
        \begin{align*}
            K_1^{\circ} = \{ \by\in\R^n : -A^{\dagger} \by \leq \bm{0} \:\: 
            \text{and} \:\: (I-A A^{\dagger})\by=\bm{0}\},
        \end{align*}
        where $A^{\dagger} = (A^T A)^{-1} A^T$.
        \item Given $B \in \R^{l \times n}$, the polar of \label{lem_cone_2b}
        \begin{align*}
            K_B 
            := 
            \{ \by\in\R^n : \ba_j^T \by \leq 0, \:\: \forall j=1,\ldots,p\} 
            \cap 
            \operatorname{nul}(B^T),
        \end{align*}
        is given by
        \begin{align*}
            K_B^{\circ} 
            = 
            \{\by\in\R^n : \by = A\bm{\lambda} + B\br, 
            \: \forall \bm{\lambda} \geq \bm{0} \: \text{and} \: \br\in\R^l\}.
        \end{align*}
    \end{enumerate}
\end{lemma}
\begin{proof}
    Part~\ref{lem_cone_1c} is \cite[Eq.~(2.8.3)--(2.8.4)]{Stoer1970}. 
    Part~\ref{lem_cone_2a} is \cite[Theorem 4.2]{Dobler1994}. 
    Part~\ref{lem_cone_2b} is \cite[Eq.~(6)]{Dobler1994}.
\end{proof}

We will use Lemma~\ref{lem_cone_technical} to characterize $\Lambda$-PSSs 
associated with $\bx \in \Omega$ and $\alpha>0$ under the following linear
independence assumption.

\begin{assumption} \label{ass_full_rank_normal}
    The set of nearly-active constraints 
    $I(\bx,\alpha)$~\eqref{eq_lincons_nearly_active} has at most 
    $n$ vectors, and the vectors $\{\ba_i : i\in I(\bx,\alpha)\}$ 
    are linearly independent.
\end{assumption}

Note that Assumption~\ref{ass_full_rank_normal} holds whenever all $\ba_i$s 
are linearly independent (in which case $m \le n$). Without loss of 
generality, and for simplicity, we assume that $I(\bx,\alpha)=\{1,\dots,q\}$ 
in the rest of this section.

\begin{lemma} \label{lem_lincons_tangent_gens}
    Suppose $\bx\in\Omega$, $\alpha>0$ and 
    Assumption~\ref{ass_full_rank_normal} holds. Then a set of generators for 
    $T_{\Omega(\bx,\alpha)}$ is the set of $2n-q$ vectors
    \begin{align*}
        \{-(A^{\dagger})^T \be_i : i=1,\ldots,q\} 
        \cup 
        \{\pm \bu_{q+1}, \ldots, \pm\bu_{n}\},
    \end{align*}
    where $A=[\ba_1 \cdots \ba_q]\in\R^{n\times q}$ and 
    $\bu_1,\ldots,\bu_{n}\in\R^n$ are the (orthonormal) left singular vectors 
    of $A$.
\end{lemma}
\begin{proof}
    By definition of $N_{\Omega}(\bx,\alpha)$ \eqref{eq_lincons_normal_cone} 
    and Lemma~\ref{lem_cone_technical}\ref{lem_cone_1c}, we may write 
    \begin{align*}
        T(\bx,\alpha) 
        = N(\bx,\alpha)^{\circ} 
        = \{ \by\in\R^n : \ba_i^T\by \leq 0, \:\: i=1,\ldots,q\},
    \end{align*}
    recalling that we are assuming $I(\bx,\alpha)=\{1,\ldots,q\}$s.
    Hence by Lemma~\ref{lem_cone_technical}\ref{lem_cone_2a},
    \begin{align*}
        N(\bx,\alpha) 
        = T(\bx,\alpha)^{\circ} 
        = \{\by\in\R^n : -A^{\dagger} \by \leq \bm{0} \:\: 
        \text{and} \: \: (I - A A^{\dagger})\by=\bm{0}\}.
    \end{align*}
    So, applying Lemma~\ref{lem_cone_technical}\ref{lem_cone_2b}, a 
    generating set for $T(\bx,\alpha) = N(\bx,\alpha)^{\circ}$ is the set of 
    columns of $-(A^{\dagger})^T$ together with any vectors that positively 
    span $\operatorname{col}(I-A A^{\dagger})$.

    It remains to show that $\{\pm \bu_{q+1}, \ldots, \pm\bu_{n}\}$ positively 
    spans $\operatorname{col}(I-A A^{\dagger})$. Since $A A^{\dagger}$ is the 
    orthogonal projector onto $\operatorname{col}(A)$, the matrix 
    $I-A A^{\dagger}$ is the orthogonal projector onto 
    $\operatorname{col}(A)^{\perp}$.
    Since $\operatorname{rank}(A)=q$ by Assumption~\ref{ass_full_rank_normal}, 
    the left singular vectors $\{\bu_1,\ldots,\bu_q\}$ form an orthonormal 
    basis for $\operatorname{col}(A)$, and $\{\bu_{q+1},\ldots,\bu_n\}$ is an 
    orthonormal basis for $\operatorname{col}(A)^{\perp}$.
    Thus $\{\pm\bu_{q+1},\ldots,\pm\bu_n\}$ positively span 
    $\operatorname{col}(A)^{\perp}$, as claimed.
\end{proof}


Given Lemma~\ref{lem_lincons_tangent_gens}, we consider the set of polling 
directions 
\begin{align}
    \mathcal{D} 
    = 
    \{\bd_1,\ldots,\bd_q\} 
    \cup \{-\alpha_1 \bd_1, \ldots, -\alpha_q \bd_q\} 
    \cup \{\pm \alpha \bu_{q+1}, \ldots, \pm \alpha \bu_n\}, 
    \label{eq_poll_set_lincons}
\end{align}
where $\bd_i := \frac{\alpha}{\|\hat{\bd}_i\|} \hat{\bd}_i$ for 
$\hat{\bd}_i := -(A^{\dagger})^T \be_i$, the constants $\alpha_i$ are the 
largest value in $[0,1]$ so that $\bx-\alpha_i \bd_i \in \Omega$, and 
$\bu_{q+1},\ldots,\bu_n$ are the last $n-q$ left singular vectors of $A$.
Note that $\hat{\bd}_i \neq \bm{0}$ since 
$\operatorname{rank}((A^{\dagger})^T) = \operatorname{rank}(A)=q$.

We first confirm that our polling set~\eqref{eq_poll_set_lincons} only 
comprises feasible points, and derive an explicit expression for the 
scaling factor $\alpha_i$ for the directions $-\bd_i$.

\begin{lemma} \label{lem_lincons_pss_feasible}
    Let Assumption~\ref{ass_full_rank_normal} hold for $\bx\in\Omega$ 
    and $\alpha>0$, and consider the polling set $\mathcal{D}$ defined by
    \eqref{eq_poll_set_lincons}.
    Then $\bx+\bd\in\Omega$ for $\bd\in\mathcal{D}$, and
    \begin{align}
        \alpha_i = \min\left(\frac{\|\hat{\bd}_i\| s_i}{\alpha}, 1\right), 
        \label{eq_linear_alpha_i}
    \end{align}
    for all $i=1,\ldots,q$, where 
    $s_i := b_i - \ba_i^T \bx \geq 0$ is the slack in the $i$-th constraint at 
    $\bx$.
\end{lemma}
\begin{proof}
    We first note that $\|\bd\| \leq \alpha$ for all $\bd\in\mathcal{D}$ by 
    construction. So, since $\bx\in\Omega$, all points $\bx+\bd$ are feasible 
    with respect to any constraints not in $I(\bx,\alpha)$ (which is defined 
    to be the set of constraints whose boundaries intersect $B(\bx,\alpha)$).
    Hence we only need to consider feasibility with respect to the constraints 
    $\ba_j^T \bx \leq b_j$ for $j=1,\ldots,q$.
    
    For all $i,j=1,\ldots,q$, we observe
    \begin{align}
        \bd_i^T \ba_j 
        = 
        -\frac{\alpha}{\|\hat{\bd}_i\|} ((A^{\dagger})^T \be_i)^T \ba_j 
        = -\frac{\alpha}{\|\hat{\bd}_i\|} \be_i^T A^{\dagger} A \be_j 
        = \begin{cases} 
        	-\frac{\alpha}{\|\hat{\bd}_i\|}, & i=j, \\ 
        	0, & i\neq j, 
        \end{cases} 
        \label{eq_lincons_orthogonality}
    \end{align}
    where the last equality follows from $A^{\dagger} A=I$ (since $A$ has full 
    column rank, Assumption~\ref{ass_full_rank_normal}). Thus, for all 
    $i,j=1,\ldots,q$ we have $\bd_i^T \ba_j \leq 0$ and so
    \begin{align*}
        \ba_j^T (\bx + \bd_i) \leq \ba_j^T \bx \leq b_j,
    \end{align*}
    where the last inequality follows from $\bx\in\Omega$.
    Hence $\bx+\bd_i\in \Omega$ for all $i=1,\ldots,q$.
    
    By construction of $\alpha_i$, we automatically have 
    $\bx-\alpha_i \bd_i\in\Omega$ for all $i=1,\ldots,q$.
    In particular, this means that $\ba_j^T (\bx - \alpha_i \bd_i) \leq b_j$, 
    or
    \begin{align*}
        s_j 
        = b_j - \ba_j^T \bx 
        \geq -\alpha_i \ba_j^T \bd_i 
        = 
        \begin{cases} 
        	\frac{\alpha_i \alpha}{\|\hat{\bd}_i\|}, & i=j, \\ 
        	0, & i\neq j, 
        \end{cases}
    \end{align*}
    from \eqref{eq_lincons_orthogonality}. 
    Hence, only the $i$-th constraint affects the feasibility of 
    $\bx-\alpha_i \bd_i$, which gives us \eqref{eq_linear_alpha_i} (noting 
    that $\alpha_i \leq 1$ by construction).

    Lastly, since 
    $\operatorname{col}(A)=\operatorname{span}(\{\bu_1,\ldots,\bu_q\})$ and 
    $\{\bu_1,\ldots,\bu_n\}$ is an orthonormal set, each of the vectors 
    $\bu_{q+1},\ldots,\bu_n$ are orthogonal to all of $\ba_1,\ldots,\ba_q$, 
    and so
    \begin{align*}
        \ba_j^T (\bx \pm \alpha \bu_i) = \ba_j^T \bx \leq b_j,
    \end{align*}
    for all $j=1,\ldots,q$ and all $i=q+1,\ldots,n$.
    Hence $\bx\pm \alpha\bu_i \in\Omega$ for all $i=q+1,\ldots,n$.
\end{proof}

We now show that our polling set forms a $\Lambda$-PSS for 
$B(\bx,\alpha)\cap \Omega$. Since $\|\bd\| \leq \alpha$ for all 
$\bd\in\mathcal{D}$, this implies that $\mathcal{D}$ satisfies 
Assumption~\ref{ass_poll_generators_lincons} with $d_{\max}=1$.

\begin{theorem} \label{thm_lincons_pss}
	Let Assumption~\ref{ass_full_rank_normal} hold for $\bx\in\Omega$ 
    and $\alpha>0$, and consider the polling set $\mathcal{D}$ defined by
    \eqref{eq_poll_set_lincons}. Then $\mathcal{D}$ is a $\Lambda$-PSS for 
    $B(\bx,\alpha)\cap \Omega$ with 
    \begin{equation}
    \label{eq_lincons_pss}
    	\Lambda 
        \; = \; 
        q \kappa(A) + \sqrt{n-q}.
    \end{equation}
\end{theorem}
\begin{proof}
	Notice first that Lemma~\ref{lem_lincons_pss_feasible} ensures 
	$\bx+\bd\in\Omega$ for all $\bd\in\mathcal{D}$.
    
    Now, fix $\bv\in\R^n$ with $\bx+\bv\in\Omega$ and $\|\bv\| \leq \alpha$. 
    We write the orthogonal decomposition $\bv=\bv_1+\bv_2$ where 
    $\bv_1 = \proj_{\operatorname{col}(A)}(\bv)$ and 
    $\bv_2=\proj_{\operatorname{col}(A)^{\perp}}(\bv)$.
    Defining the slack variables $s_j$ as in 
    Lemma~\ref{lem_lincons_pss_feasible}), it follows that 
    $\ba_j^T \bv = \ba_j^T \bv_1 \leq s_j$ for all $j=1,\ldots,q$. 
    Moreover, since $\|\bv_1\|^2 + \|\bv_2\|^2 = \|\bv\|^2 \leq \alpha^2$, we 
    have $\max\{\|\bv_1\|,\|\bv_2\|\} \le \alpha$.
    
    We begin by consider $\bv_2$. The definition of
    $\operatorname{col}(A)^{\perp} = 
    \operatorname{span}(\{\bu_{q+1},\ldots,\bu_n\})$ gives
    \begin{align*}
        \bv_2 
        = \sum_{i=q+1}^{n} (\bv_2^T \bu_i) \bu_i 
        = \sum_{i=q+1}^{n} \frac{\bv_2^T \bu_i}{\alpha} \: \alpha \bu_i 
        = \sum_{i=q+1}^{n} \frac{|\bv_2^T \bu_i|}{\alpha} \: (\pm \alpha \bu_i).
    \end{align*}
    Since $|\bv_2^T \bu_i| \leq \|\bv_2\| \: \|\bu_i\| \leq \alpha$, we can 
    therefore write
    \begin{align*}
        \bv_2 = \sum_{i=q+1}^{n} c_i (\alpha \bu_i) + c_{-i} (-\alpha \bu_i),
    \end{align*}
    where $c_i=$ and $c_{-i}$ are defined so the proof of 
    Theorem~\ref{thm_pss_bounds}. Since the vectors $\pm \bu_i$ are feasible 
    for $\alpha$, applying the reasoning of Theorem~\ref{thm_pss_bounds} 
    guarantees that 
    \begin{equation}
    \label{eq:sumciv2}
    	\sum_{i=q+1}^n (c_i+c_{-i}) \le \sqrt{n-q}.
    \end{equation}

    Consider now the vector 
    $\bv_1\in \operatorname{col}(A) = \operatorname{col}((A^{\dagger})^T)$.
   	Since the linear system $-(A^{\dagger})^T \hat{\bc}=\bv_1$ is consistent, 
   	the vector 
    $\hat{\bc}=(-(A^{\dagger})^T)^{\dagger}\bv_1 = -A^T \bv_1 \in \R^q$ is its 
    minimal norm solution. It follows that
    \begin{align}
        \bv_1 
        = \sum_{i=1}^{q} \hat{c}_i \hat{\bd}_i 
        = \sum_{i=1}^{q} c_i \bd_i, 
        \qquad \text{where} \quad 
        c_i := \frac{\hat{c}_i \|\hat{\bd}_i\|}{\alpha}, 
        \label{eq_lincons_pss_tmp1}
    \end{align}
    with $\hat{c}_i,c_i\in\R$, possibly negative.
    Since $\hat{c}_i = [A^T \bv_1]_i$ and 
    $\hat{\bd}_i = -(A^{\dagger})^T \be_i$, we have
    \begin{align}
        |c_i| 
        \le 
        \frac{\|A^T \bv_1\|_{\infty} \|(A^{\dagger})^T\|}{\alpha} 
        \le 
        \frac{\|A^T \bv_1\| \|(A^{\dagger})^T\|}{\alpha} \leq \kappa(A), 
        \label{eq_lincons_pss_tmp2}
    \end{align}
    where the last inequality uses $\|\bv_1\| \leq \alpha$.

    From above, we know that each of $\bu_{q+1},\ldots,\bu_n$ are orthogonal 
    to all of $\ba_1,\ldots,\ba_q$, and so $\bv_2$ is orthogonal to all of 
    $\ba_1,\ldots,\ba_q$. Hence,
    \begin{align*}
        \ba_j^T \bv_1 = \ba_j^T \bv \leq s_j
    \end{align*}
    for all $j=1,\ldots,q$.
    From \eqref{eq_lincons_orthogonality} and \eqref{eq_lincons_pss_tmp1}, 
    we get
    \begin{align}
        s_j 
        \ge \ba_j^T \bv_1 
        = \sum_{i=1}^{q} c_i \bd_i^T \ba_j 
        = -\frac{\alpha c_j}{\|\hat{\bd}_j\|}, 
        \label{eq_lincons_pss_tmp3}
    \end{align}
    for all $j=1,\ldots,q$. This implies that $c_i\ge 0$ whenever $s_i=0$ 
    (i.e.~$\bx$ lies on the boundary of the $i$-th constraint). Conversely, if 
    $c_i<0$ then $s_i > 0$ and so $\alpha_i > 0$ 
    from~\eqref{eq_linear_alpha_i} (recalling that 
    $\hat{\bd}_i \neq \bm{0}$).

    Since we do not know which (if any) $c_i$ in~\eqref{eq_lincons_pss_tmp1} 
    are non-negative, assume without loss of generality that 
    $c_1,\ldots,c_k<0$ and $c_{k+1},\ldots,c_q \geq 0$ for some 
    $k\in\{0,\ldots,q\}$. In that case, $\alpha_1,\ldots,\alpha_k>0$, and so 
    we can write
    \begin{align*}
        \bv_1 
        = 
        \sum_{i=1}^{k} \frac{|c_i|}{\alpha_i} (-\alpha_i \bd_i) 
        + \sum_{i=k+1}^{q} |c_i| \bd_i.
    \end{align*}
    This expresses $\bv_1$ as a non-negative combination of polling 
    directions. We already have an upper bound on 
    $|c_i|$~\eqref{eq_lincons_pss_tmp2}, so it remains to bound 
    $|c_i|/\alpha_i$ for $i=1,\ldots,k$. For any $i=1,\ldots,k$, we 
    use~\eqref{eq_lincons_orthogonality} to get
    \begin{align*}
        \ba_i^T \bv_1 
        = \frac{|c_i|}{\alpha_i} (-\alpha_i \ba_i^T \bd_i) 
        = \frac{|c_i|}{\alpha_i} \cdot \frac{\alpha_i \alpha}{\|\hat{\bd}_i\|} 
        \quad \Leftrightarrow \quad 
        \frac{|c_i|}{\alpha_i} 
        = \frac{\ba_i^T \bv_1 \|\hat{\bd}_i\|}{\alpha_i \alpha} 
        \le \frac{s_i \|\hat{\bd}_i\|}{\alpha_i \alpha},
    \end{align*}
    where the last inequality uses $\ba_i^T \bv_1 \le s_i$ 
    from~\eqref{eq_lincons_pss_tmp3}. By definition of $\alpha_i$ 
    in~\eqref{eq_linear_alpha_i}, it follows that
    \begin{align*}
        \frac{|c_i|}{\alpha_i} 
        \le \frac{s_i \|\hat{\bd}_i\|}{\alpha_i \alpha} 
        = 1
        \quad \mbox{if}\ \alpha_i=\frac{s_i \|\hat{\bd}_i\|}{\alpha}
        \qquad \mbox{and} \qquad 
        \frac{|c_i|}{\alpha_i} = |c_i| \le \kappa(A) 
        \quad \mbox{if}\ \alpha_i=1,
    \end{align*}
    where the second inequality uses~\eqref{eq_lincons_pss_tmp2}.
	Since $\kappa(A) \ge 1$ by definition, summing over $i=1,\dots,k$ yields
	\begin{equation}
	\label{eq:cPicMisum1k}    
    	\sum_{i=1}^k \frac{|c_i|}{\alpha_i}
    	\le k\,\kappa(A). 
    \end{equation}
    Finally, for $i=k+1,\dots,q$, using~\eqref{eq_lincons_pss_tmp2} gives
    \begin{equation}
    \label{eq:cPicMisumkp1q}
    	\sum_{i=k+1}^q |c_i| \le (q-k) \kappa(A).
    \end{equation}
    Overall, we have decomposed $\bv$ into
    \begin{align*}
        \bv 
        = \bv_1 + \bv_2 
        = \sum_{i=1}^{k} \frac{|c_i|}{\alpha_i} (-\alpha_i \bd_i) 
        + \sum_{i=k+1}^{q} |c_i| \bd_i 
        + \sum_{i=q+1}^{n} \left[c_i (\alpha \bu_i) 
        + c_{-i} (-\alpha \bu_i)\right],
    \end{align*}
    where all coefficients are non-negative, and their sum is bounded by
    \begin{align*}
        \sum_{i=1}^{k} \frac{|c_i|}{\alpha_i} 
        + \sum_{i=k+1}^{q} |c_i| 
        + \sum_{i=q+1}^{n} (c_i + c_{-i}) 
        \le 
        k \kappa(A) + (q-k) \kappa(A) + \sqrt{n-q}
        =
        q \kappa(A) + \sqrt{n-q},
    \end{align*}
    which gives the value in~\eqref{eq_lincons_pss}.
\end{proof}

The above bound gives us a value of $\Lambda$ depending on 
$A=[\ba_1 \cdots \ba_q]$, and so this value depends on the particular set of 
nearly active constraints, $I(\bx,\alpha)$.  In particular, when there are no 
active constraints ($q=0$), we recover the unconstrained value $\Lambda = \sqrt{n}$.
We now adjust the bound to make it 
completely independent of $\bx$ and $\alpha$.

\begin{corollary}
\label{coro:bndkappanC}
    Suppose the assumptions of Theorem~\ref{thm_lincons_pss} hold.
    Then, there exists a constant $C>0$ depending only on the constraint 
    vectors $\ba_1,\ldots,\ba_m$ \eqref{eq_linear_cons} such that 
    $\mathcal{D}$ is a $\Lambda$-PSS for $B(\bx,\alpha)\cap \Omega$ where 
    $\Lambda \le nC$.
\end{corollary}
\begin{proof}
	For any $(\bx,\alpha)$, we can apply Theorem~\ref{thm_lincons_pss} with 
	a matrix $A_I=\{\ba_i : i \in I\}$ to get 
    \begin{align*}
    	\Lambda = |I|\kappa(A_I) + 
    	\sqrt{n-|I|} \le n \kappa(A_I),
    \end{align*}
    where the second inequality comes from taking a maximum over 
    $|I|\in\{0,\ldots,n\}$ with $\kappa(A_I)\geq 1$.
    Taking the maximum over the finitely many possibilities for $I$ then 
    yields the desired result.
\end{proof}

Since $\mathcal{D}$ has $2n$ vectors with $\Lambda=\bigO(n)$, 
Corollary~\ref{coro:bndkappanC} together with Theorem~\ref{thm_wcc_convex_new} 
imply that Algorithm~\ref{alg_ds_basic_new} has a worst-case complexity of 
$\bigO(n^2 \epsilon^{-2})$ iterations and $\bigO(n^3 \epsilon^{-2})$ 
evaluations to achieve first-order optimality $\pi(\bx_k) \leq \epsilon$. To 
the best of our knowledge, this is the first bound with quantifiable dependency 
on $n$ for direct search with linear inequality constraints other than 
bounds.

\subsection{Handling general linear constraints in practice} 
\label{sec_practical_poll_set}

In the previous section, we provided explicit constructions for polling sets 
for bound-constrained problems and linearly constrained 
problems~\eqref{eq_linear_cons} satisfying 
Assumption~\ref{ass_full_rank_normal}. Assuming Asssumption~\ref{ass_omega} 
but not Assumption~\ref{ass_full_rank_normal}, we do not know of a polling set 
construction with guaranteed $\Lambda$-PSS properties. However, various 
algorithms exist for constructing a (minimal) set of generators for a general 
cone expressed as linear inequalities, such as the double description 
method~\cite{Fukuda1996}. In practice, we can use such a technique to 
construct a set of generators for $T_{\Omega}(\bx,\alpha)$ for any set of 
linear inequality constraints.

To this end, a practical polling set construction method given $\bx\in\Omega$ 
and $\alpha>0$ consists in the following steps:
\begin{enumerate}
    \item If $I(\bx,\alpha)=\emptyset$, take 
    $\mathcal{D} = \{\pm \be_1, \ldots, \pm \be_n\}$ (as in the unconstrained 
    case).
    \item If Assumption~\ref{ass_full_rank_normal} holds, take $\mathcal{D}$ 
    as in \eqref{eq_poll_set_lincons}.
    \item Otherwise, construct a set of generators $\mathcal{G}$ for 
    $T_{\Omega}(\bx,\alpha)$.
    \begin{enumerate}
        \item If $\mathcal{G}\neq\emptyset$, set 
        $\mathcal{D} = \{\frac{\alpha}{\|\bd\|} \bd : \bd \in \mathcal{G}\} 
        \cup \{-\alpha_d \bd : \bd \in \mathcal{G}\}$, where $\alpha_x$ is 
        the largest value in $[0,\alpha]$ such that 
        $\bx-\alpha_d \bd \in \Omega$;
        \item Otherwise, if $\mathcal{G} = \emptyset$ 
        (i.e.~$T_{\Omega}(\bx,\alpha)=\{\bm{0}\})$ then take $\mathcal{D}$ to 
        be the generators of $N_{\Omega}(\bx,\alpha)$. 
        More precisely, set 
        $\mathcal{D} = \{\alpha_i \ba_i : i\in I(\bx,\alpha)\}$, where 
        $\alpha_i\geq 0$ is the largest value such that the polled point is 
        feasible and in $B(\bx,\alpha)$.
        Since $N_{\Omega}(\bx,\alpha)=\R^n$ in this case, this set is a PSS 
        for $\R^n$.
        \item If $\operatorname{span}(\mathcal{D}) \neq \R^n$ (linear span) 
        for $\mathcal{D}$ from 3(a), then append to $\mathcal{D}$ a polling 
        set for $\operatorname{nul}(\mathcal{G})\cap \Omega$ constructed 
        recursively using this approach.
    \end{enumerate}
\end{enumerate}

Although this approach is recursive (in the last step), it always terminates 
in finite time, since every recursive call operates on a proper subspace of 
its parent call. Note that if $\operatorname{nul}(\mathcal{G})$ is 
one-dimensional, then we just take positive and negative steps (scaled to 
ensure feasibility and length at most $\alpha$) along this one direction to 
complete the recursive step.

\section{Numerical Experiments} 
\label{sec_numerics}

In this section, we compare the strategy described in 
Section~\ref{sec_poll_sets} with classical direct-search variants tailored to 
bound and linearly constrained problems. In addition to comparing these 
methods on a standard optimization benchmark, we also evaluate the quality of 
polling sets when generated as $\Lambda$-PSSs.

\subsection{Implementation Details}
\label{sec_implementation}

Our implementation relies on the \texttt{directsearch} \cite{Roberts2023} 
Python package\footnote{\url{https://github.com/lindonroberts/directsearch}}, 
which we augment with three different polling techniques to handle linear 
inequality constraints. More precisely, we implement 
Algorithm~\ref{alg_ds_basic_new} with three choices of polling sets:
\begin{itemize}
    \item Tangent generators: At iteration $k$, $\mathcal{D}_k$ is the set of 
    generators of the tangent cone, scaled to have length $\alpha_k$. If this 
    set is empty (i.e.~$T_{\Omega}(\bx_k,\alpha_k)=\{\bm{0}\}$) then we use 
    the scaled generators of the normal cone, as in step 3(b) of 
    Section~\ref{sec_practical_poll_set}.
    \item Tangent and normal generators: At iteration $k$, $\mathcal{D}_k$ is 
    the set of generators of the tangent cone, scaled to have length 
    $\alpha_k$, together with the nearly active constraints normals, scaled to 
    be feasible and have length at most $\alpha_k$ (as in step 3(b) of 
    Section~\ref{sec_practical_poll_set}). This follows the heuristic approach 
    from Lewis et al.~\cite{Lewis2007}.
    \item Full $\Lambda$-PSS: the full polling set generation procedure 
    described in Section~\ref{sec_practical_poll_set}. The extra `negative 
    directions' (i.e.~$-\alpha_d \bd$ for $\bd\in\mathcal{G}$) are only polled 
    after the directions in the tangent cone.
\end{itemize}

All variants consider the same generators of the tangent cone, and the 
implementation always polls tangent directions first. We use opportunistic
polling\footnote{That is, the first $\bd_k\in\mathcal{D}_k$ satisfying the 
sufficient decrease condition is accepted.} and we replace the quantity 
$\tfrac{\sigma}{2}\alpha_k^2$ in~\eqref{eq_sufficient_decrease} 
(resp.~\eqref{eq_sufficient_decrease_new}) with the slightly modified 
quantity $\min(\epsilon, \epsilon \alpha_k^2)$ where $\epsilon=10^{-5}$. 
We update $\alpha_k$ with $\gammadec=0.5$, $\gammainc=2$ and with $\alpha_k$ 
capped at $\alpha_{\max}=10^3$. We terminate when 
$\alpha_k \leq \alpha_{\min}=10^{-6}$ or if a budget of $200(n+1)$ objective 
evaluations is reached for an $n$-dimensional problem. The initial stepsize 
was chosen as $\alpha_0 = \max(\alpha_{\min}, 
\max(0.1\max(\|\bx_0\|_{\infty}, 1), \alpha_{\max}))$.

For this implementation, we test the two polling set constructions on a 
collection of 122 problems from the CUTEst collection~\cite{Gould2015,
Fowkes2022}. These problems are mostly low-dimensional, with dimensions 
between $1$ and $51$. Bound constraints appear in 77 problems, while 45 
problems possess general linear inequality constraints. This collection of 
test problems is based on the problems used for numerical experiments in 
Gratton et al.~\cite{Gratton2019}, and a full list of problems is given in 
Appendix~\ref{app_test_problems}. Where the starting point $\bx_0$ provided 
by CUTEst is not feasible, we replace it with $\proj_{\Omega}(\bx_0)$, to 
ensure the starting point is feasible.

\subsection{Main Results}
\label{sec_main_resnum}

We report our comparison using data and performance profiles~\cite{More2009,
Dolan2002}. Specifically, for each solver $\mathcal{S}$ and problem 
$\mathcal{P}$ we calculate the number of evaluations required for the problem 
to be `solved' as
\begin{align}
    N(\mathcal{S},\mathcal{P},\tau)
    := 
    \text{\# obj.~evals.~to find $\bx\in\Omega$ with 
    $f(\bx) \leq f_{\min} + \tau (f(\bx_0) - f_{\min})$}, 
    \label{eq_numerics_solved}
\end{align}
where $\tau\ll 1$ is an accuracy parameter, and $f_{\min}$ is the best 
(feasible) objective value found by any solver for the given problem. We 
define $N(\mathcal{S},\mathcal{P},\tau) = \infty$ (i.e.~solver $\mathcal{S}$ 
never solved problem $\mathcal{P}$) if the solver never found a feasible point 
with sufficiently small objective value. Data profiles measure the proportion 
of problems $\mathcal{P}$ `solved' by a solver $\mathcal{S}$ (in the sense of 
\eqref{eq_numerics_solved}) within a given number of evaluations $c(n+1)$ for 
some constant $c>0$. Performance profiles measure the proportion of problems 
solved within some constant of the fastest solver for that problem, 
i.e.~$N(\mathcal{S},\mathcal{P},\tau) 
\leq c \min_{\mathcal{S}'} N(\mathcal{S}', \mathcal{P}, \tau)$ for some 
constant $c\geq 1$.

\begin{figure}[tb]
  \centering
  \begin{subfigure}[b]{0.48\textwidth}
    \includegraphics[width=\textwidth]{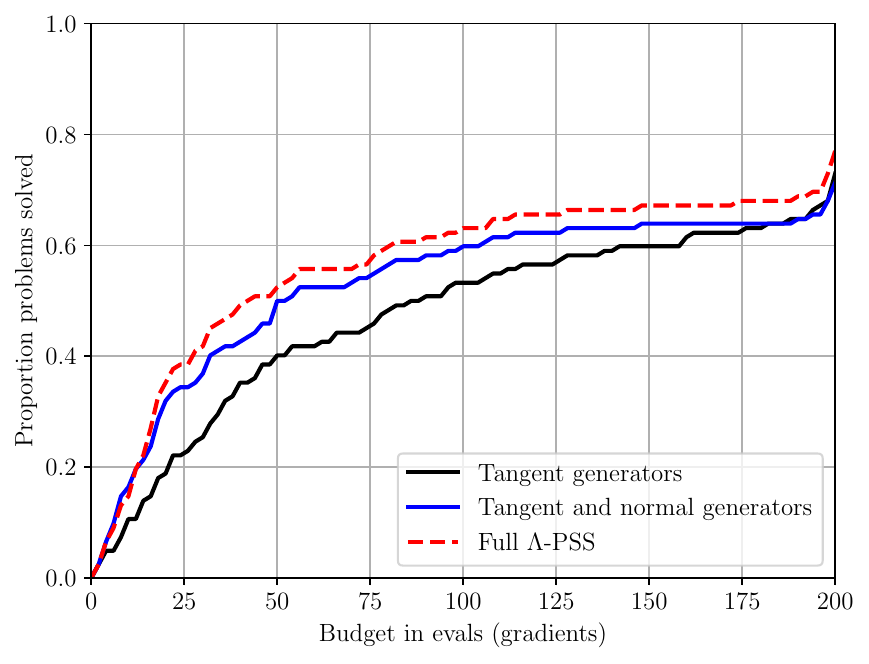}
    \caption{Data profile, $\tau=10^{-3}$}
  \end{subfigure}
  ~
  \begin{subfigure}[b]{0.48\textwidth}
    \includegraphics[width=\textwidth]{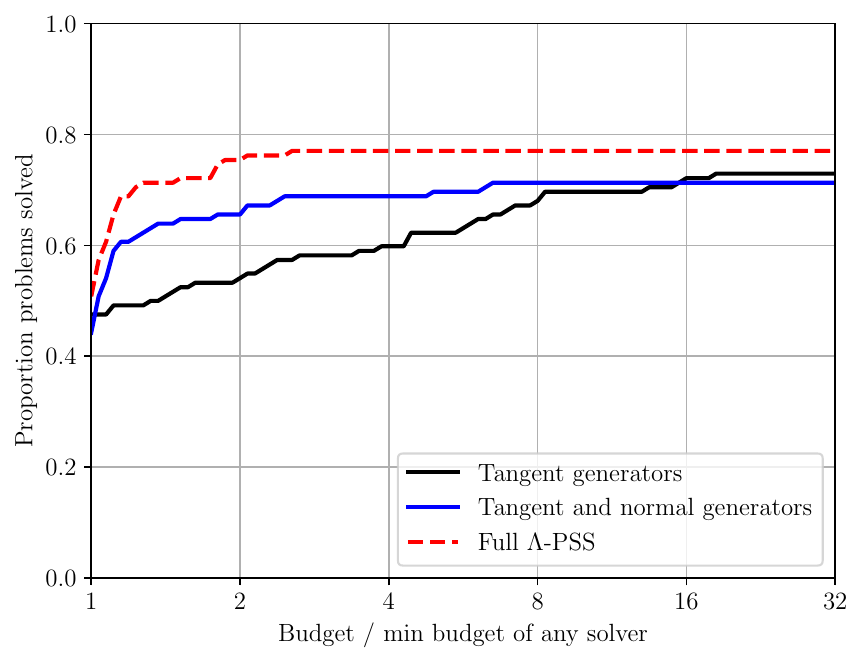}
    \caption{Performance profile, $\tau=10^{-3}$}
  \end{subfigure}
  \\
  \begin{subfigure}[b]{0.48\textwidth}
    \includegraphics[width=\textwidth]{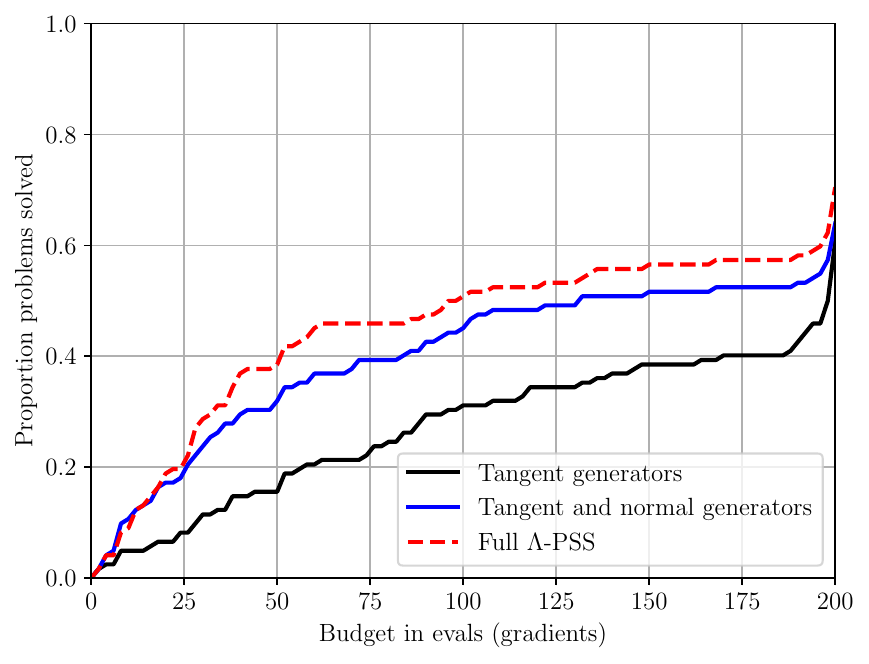}
    \caption{Data profile, $\tau=10^{-6}$}
  \end{subfigure}
  ~
  \begin{subfigure}[b]{0.48\textwidth}
    \includegraphics[width=\textwidth]{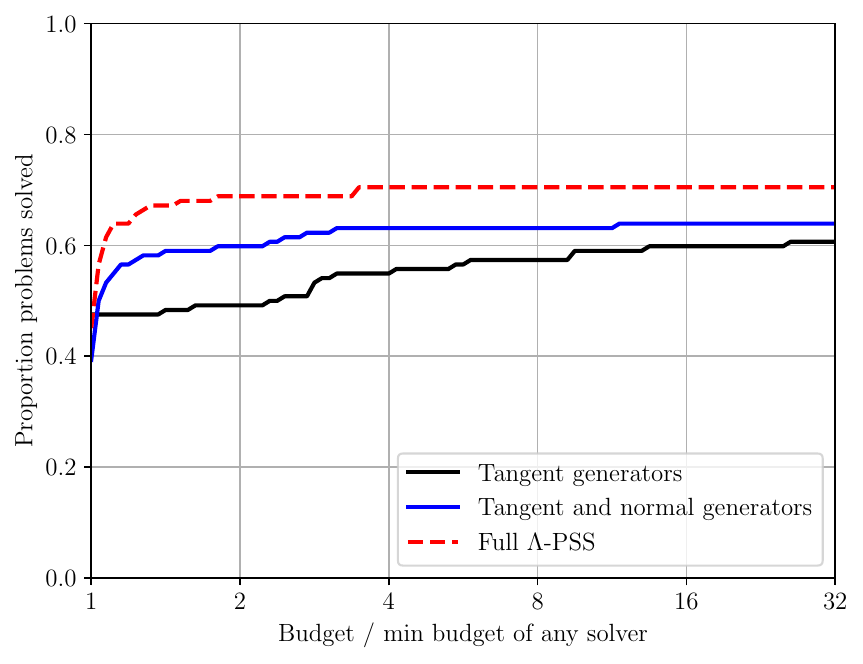}
    \caption{Performance profile, $\tau=10^{-6}$}
  \end{subfigure}
  \caption{Data and performance profiles comparing polling set generation 
  using only tangent cone generators compared to full $\Lambda$-PSS.}
  \label{fig_main_results}
\end{figure}

In Figure~\ref{fig_main_results} we show data and performance profiles 
comparing the three methods for the full set of test problems. For both 
accuracy levels $\tau=10^{-3}$ and $\tau=10^{-6}$, we observe that only using 
generators of the tangent cone is the worst-performing variant. Adding nearly 
active constraint normals (per \cite{Lewis2007}) provides a good improvement, 
but using a full $\Lambda$-PSS (i.e.~adding negative tangent cone generators) 
performs best. The benefit of $\Lambda$-PSS compared to tangent generators and 
normal directions is similar at both accuracy levels, but using only tangent 
directions is a relatively worse option when higher accuracy solutions are 
desired ($\tau=10^{-6}$).

In Appendix~\ref{app_detailed_numerics} we provide the same results, but split 
separately into bound constrained and general linear inequality constrained 
problems. For bound-constrained problems, using either normal generators or a 
full $\Lambda$-PSS has essentially identical performance, which is somewhat 
expected from our construction. The benefit of using a $\Lambda$-PSS is much 
clearer for general linear inequality constrained problems, yielding a larger 
improvement over both alternative polling set constructions.

\subsection{Estimating $\bm{\Lambda}$-PSS Quality} 
\label{sec_estimated_lambda}

The ``full $\Lambda$-PSS'' polling set tested above is only proven to satisfy 
Definition~\ref{def_lambda_pss_convex_new} for general linear inequality 
constraints in the case of linearly independent tangent cone generators 
(Assumption~\ref{ass_full_rank_normal}). We now numerically assess the quality 
of the polling set formed via the approach in 
Section~\ref{sec_practical_poll_set}.

For all test problems with general linear inequality constraints, and all 
iterations of our numerical algorithm above, we estimate the value of 
$\Lambda$ (in the sense of Definition~\ref{def_lambda_pss_convex_new}) 
achieved by the polling set in the current iteration, by solving
\begin{subequations} \label{eq_lambda_est}
\begin{align}
    \Lambda 
    = 
    \max_{\bv\in\R^n} &\:\: 
    	\left\{\min_{\bc\in\R^p} \sum_{i=1}^{q} c_i 
    	\quad \text{s.t.} \quad \sum_{i=1}^{q} c_i \bd_i = \bv, 
    	\: \bc\geq \bm{0}\right\}, \\
    \text{s.t.} 
    	&\:\: \bx_k+\bv\in \Omega, \: \|\bv\| \leq \alpha_k, 
\end{align}
\end{subequations}
where $\mathcal{D}_k = \{\bd_1, \ldots,\bd_q\}$ is the polling set at 
iteration $k$. We solve \eqref{eq_lambda_est} using the default implementation 
of DIRECT provided in the SciPy library (based on \cite{Gablonsky2001}), with 
the objective value evaluated by solving the linear program for $\bc$ using 
HiGHS \cite{Huangfu2018}.

\begin{figure}[tb]
    \centering
    \includegraphics[width=0.8\textwidth]{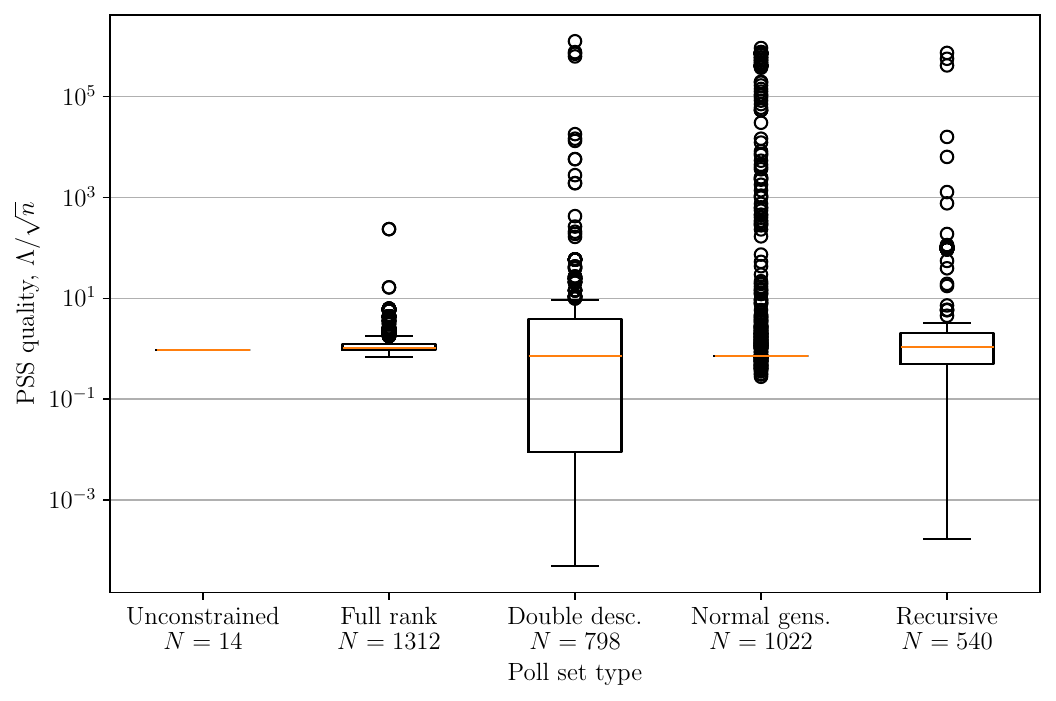}
    \caption{Distribution of estimated values $\Lambda/\sqrt{n}$ for polling 
    sets in each iteration of each general linear inequality constrained test 
    problem. Regarding the method in Section~\ref{sec_practical_poll_set}: 
    ``Unconstrained'' is case 1, ``Full rank'' is case 2, ``Double desc.'' is 
    case 3(a), ``Normal gens.'' is case 3(b), and ``Recursive'' is case 3(c). 
    In case case, the value of $N$ is how many iterations were in each case 
    across all problems.}
    \label{fig_lambda_info}
\end{figure}

In Figure~\ref{fig_lambda_info}, we plot the distribution of estimated values 
of $\Lambda/\sqrt{n}$ \eqref{eq_lambda_est} across all iterations of all test 
problems with general linear inequality constraints, split by the sub-method 
from Section~\ref{sec_practical_poll_set} used to construct the polling set.
We show $\Lambda/\sqrt{n}$ to normalize across problems of different 
dimensions, as this value is exactly 1 for unconstrained problems (with polling 
set $\{\pm \alpha_k \be_1, \ldots, \pm \alpha_k \be_n\}$).

In the `unconstrained' and `full rank' cases, we always see small values of 
$\Lambda=\bigO(\sqrt{n})$, which aligns with the theory from 
Sections~\ref{sec_newwcc_unc} and \ref{sec_poll_sets} respectively, noting 
that $n\leq 15$ for all these test problems.
For the other types of polling set constructions, the empirical values of 
$\Lambda$ are small the vast majority of the time, but there are outliers for 
which $\Lambda$ can become large.
This conforms to our theory and further demonstrates that for the cases not 
covered by our theory, the practical polling set generation in 
Section~\ref{sec_practical_poll_set} is suitable for most situations, but some 
degenerate cases may arise which require more careful polling set construction 
to ensure theoretical guarantees.

\begin{figure}[tb]
    \centering
    \includegraphics[width=0.5\textwidth]{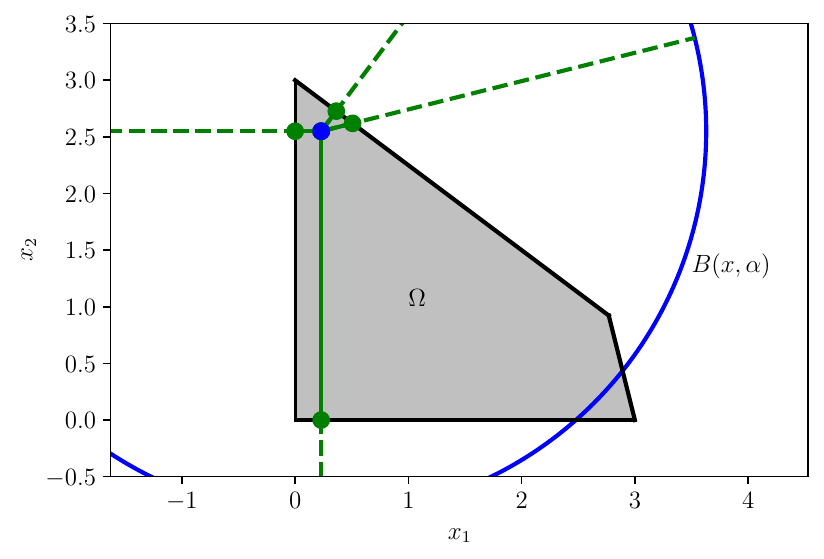}
    \caption{Example issue that can occur using generators of the normal cone 
    as a polling set in case 3(b) of Section~\ref{sec_practical_poll_set}. 
    Shaded region is $\Omega$, blue circle is $B(\bx,\alpha)$, marked points 
    are $\bx$ and the poll points based on scaled outward normals of nearly 
    active constraints. All constraints are nearly active for this $\bx$ and 
    $\alpha$.}
    \label{fig_demo_hs44}
\end{figure}

As an example, consider a situation in $\R^2$ where $\Omega$ is defined by the 
constraints $x_1,x_2 \geq 0$, $4x_1+x_2 \leq 12$ and $3x_1+4x_2 \leq 12$, and 
we are at the base point $\bx=[0.23, 2.55]^T$ with $\alpha=3.4$. This is 
depicted in Figure~\ref{fig_demo_hs44}. Here, $\bx$ is far from the constraint 
$4x_1+x_2 \leq 12$, but $\alpha$ is large enough that it is still nearly 
active. Hence $T_{\Omega}(\bx,\alpha)=\{\bm{0}\}$ and 
$N_{\Omega}(\bx,\alpha)=\R^2$. However, because $\bx$ is very close to the top 
corner, the rightward-pointing outward normals both get scaled to very short 
lengths. This means that the effective $\Lambda$ is very large, coming from 
points $\bx+\bv$ near the bottom-right corner of $B(\bx,\alpha)\cap \Omega$.
As $\bx$ moves closer to the top corner of $\Omega$, the effectively $\Lambda$ 
grows.

\section{Conclusion}
\label{sec_conc}

We have introduced the concept of a $\Lambda$-PSS as an alternative to the 
commonly-used positive spanning sets for direct search. This allowed us to 
re-derive the existing theory for unconstrained direct search, but with a 
theoretical underpinning more aligned with model-based DFO, and which 
naturally generalizes to convex-constrained problems. We provided a global 
convergence and worst-case complexity analysis for general convex-constrained 
direct search, and  specific constructions for polling set generation for 
bound and linear inequality constrained problems with theoretical guarantees 
and strong practical performance.

In presence of linearly dependent nearly active constraints (i.e.~when 
Assumption~\ref{ass_full_rank_normal} does not apply), our proposed 
construction does not guarantee the $\Lambda$-PSS property, despite exhibiting 
good practical performance. Further work is thus needed to provably build 
$\Lambda$-PSS in this setting. In addition, combining $\Lambda$-PSSs with 
probabilistic polling set construction~\cite{Gratton2019} or polling in 
randomly generated subspaces~\cite{Roberts2023} is a promising direction 
to improve the scalability and performance of these techniques.

\addcontentsline{toc}{section}{References} 
\bibliographystyle{siam}
{ 
\bibliography{refs} 
}

\appendix


\section{List of Test Problems} \label{app_test_problems}

\begin{table}[H]
    \centering
    {\scriptsize
    \begin{tabular}{ccc|ccc}
        Name (params) & $n$ & NB & Name (params) & $n$ & NB \\ \hline
        ALLINIT & 3 & 3 & MAXLIKA & 8 & 16 \\
        BQP1VAR & 1 & 2 & MCCORMCK ($N=10$) & 10 & 20 \\
        CAMEL6 & 2 & 4 & MCCORMCK ($N=50$) & 50 & 100 \\
        CHEBYQAD ($N=10$) & 10 & 20 & MDHOLE & 2 & 1 \\
        CHEBYQAD ($N=20$) & 20 & 40 & NCVXBQP1 ($N=10$) & 10 & 20 \\
        CHENHARK (*) & 10 & 10 & NCVXBQP1 ($N=50$) & 50 & 100 \\
        CVXBQP1 ($N=10$) & 10 & 20 & NCVXBQP2 ($N=10$) & 10 & 20 \\
        CVXBQP1 ($N=50$) & 50 & 100 & NCVXBQP2 ($N=50$) & 50 & 100 \\
        DEGDIAG ($N=10$) & 11 & 11 & NCVXBQP3 ($N=10$) & 10 & 20 \\
        DEGDIAG ($N=50$) & 51 & 51 & NCVXBQP3 ($N=50$) & 50 & 100 \\
        DEGTRID ($N=10$) & 11 & 11 & NOBNDTOR ($Q=2$) & 4 & 4 \\
        DEGTRID ($N=50$) & 51 & 51 & OBSTCLAE ($PX=4$, $PY=4$) & 4 & 8 \\
        EG1 & 3 & 4 & OBSTCLBL ($PX=4$, $PY=4$) & 4 & 8 \\
        EXPLIN ($N=12$, $M=6$) & 12 & 24 & OSLBQP & 8 & 11 \\
        EXPLIN2 ($N=12$, $M=6$) & 12 & 24 & PALMER1A & 6 & 2 \\
        EXPQUAD ($N=12$, $M=6$) & 12 & 12 & PALMER2B & 4 & 2 \\
        HARKERP2 ($N=10$) & 10 & 10 & PALMER3E & 8 & 1 \\
        HART6 & 6 & 12 & PALMER4A & 6 & 2 \\
        HATFLDA & 4 & 4 & PALMER5B & 9 & 2 \\
        HATFLDB & 4 & 5 & PFIT1LS & 3 & 1 \\
        HIMMELP1 & 2 & 4 & POWELLBC ($P=5$) & 10 & 20 \\
        HS1 & 2 & 1 & POWELLBC ($P=10$) & 20 & 40 \\
        HS110 ($N=10$) & 10 & 20 & PROBPENL ($N=10$) & 10 & 20 \\
        HS110 ($N=50$) & 50 & 100 & PROBPENL ($N=50$) & 50 & 100 \\
        HS2 & 2 & 1 & PSPDOC & 4 & 1 \\
        HS25 & 3 & 6 & QRTQUAD ($N=12$, $M=6$) & 12 & 24 \\
        HS3 & 2 & 1 & S368 ($N=8$) & 8 & 16 \\
        HS38 & 4 & 8 & S368 ($N=50$) & 50 & 100 \\
        HS3MOD & 2 & 1 & SCOND1LS ($N=10$, $LN=9$) & 10 & 20 \\
        HS4 & 2 & 2 & SCOND1LS ($N=50$, $LN=45$) & 50 & 100 \\
        HS45 & 5 & 10 & SIMBQP & 2 & 2 \\
        HS5 & 2 & 4 & SINEALI ($N=10$) & 10 & 20 \\
        JNLBRNG1 ($PT=4$, $PY=4$) & 4 & 4 & SINEALI ($N=20$) & 20 & 40 \\
        JNLBRNG2 ($PT=4$, $PY=4$) & 4 & 4 & SPECAN ($K=3$) & 9 & 18 \\
        JNLBRNGA ($PT=4$, $PY=4$) & 4 & 4 & TORSION1 ($Q=2$) & 4 & 8 \\
        JNLBRNGB ($PT=4$, $PY=4$) & 4 & 4 & TORSIONA ($Q=2$) & 4 & 8 \\
        KOEBHELB & 3 & 2 & WEEDS & 3 & 4 \\
        LINVERSE ($N=10$) & 19 & 10 & YFIT & 3 & 1 \\
        LOGROS & 2 & 2 &  \\
        \hline
    \end{tabular}
    }  
    \caption{List of 77 bound-constrained CUTEst problems used for numerical experiments. (* parameters for CHENHARK are $N=10$, $\text{NFREE}=5$, $\text{NDEGEN}=2$)}
    \label{tab:pbboundcons}
\end{table}

\begin{table}[H]
    \centering
    {\scriptsize
    \begin{tabular}{cccc|cccc}
        Name (params) & $n$ & NB & LI & Name (params) & $n$ & NB & LI \\ \hline
        AVGASA & 8 & 16 & 10 & HS86 & 5 & 5 & 10 \\
        AVGASB & 8 & 16 & 10 & HUBFIT & 2 & 1 & 1 \\
        BIGGSC4 & 4 & 8 & 13 & LSQFIT & 2 & 1 & 1 \\
        EQC & 7 & 14 & 3 & OET1 & 3 & 0 & 1002 \\
        EXPFITA & 5 & 0 & 22 & OET3 & 4 & 0 & 1002 \\
        EXPFITB & 5 & 0 & 102 & PENTAGON & 6 & 0 & 15 \\
        EXPFITC & 5 & 0 & 502 & PT & 2 & 0 & 501 \\
        HATFLDH & 4 & 8 & 13 & QC & 7 & 14 & 4 \\
        HS105 & 8 & 16 & 1 & QCNEW & 7 & 14 & 3 \\
        HS118 & 15 & 30 & 29 & S268 & 5 & 0 & 5 \\
        HS21 & 2 & 4 & 1 & SIMPLLPA & 2 & 2 & 2 \\
        HS21MOD & 7 & 8 & 1 & SIMPLLPB & 2 & 2 & 3 \\
        HS24 & 2 & 2 & 3 & SIPOW1 & 2 & 0 & 2000 \\
        HS268 & 5 & 0 & 5 & SIPOW1M & 2 & 0 & 2000 \\
        HS35 & 3 & 3 & 1 & SIPOW2 & 2 & 0 & 2000 \\
        HS35I & 3 & 6 & 1 & SIPOW2M & 2 & 0 & 2000 \\
        HS35MOD & 2 & 2 & 1 & SIPOW3 & 4 & 0 & 2000 \\
        HS36 & 3 & 6 & 1 & SIPOW4 & 4 & 0 & 2000 \\
        HS37 & 3 & 6 & 2 & STANCMIN & 3 & 3 & 2 \\
        HS44 & 4 & 4 & 6 & TFI2 & 3 & 0 & 101 \\
        HS44NEW & 4 & 4 & 6 & TFI3 & 3 & 0 & 101 \\
        HS76 & 4 & 4 & 3 & ZECEVIC2 & 2 & 4 & 2 \\
        HS76I & 4 & 8 & 3 &  &  &  &  \\
        \hline
    \end{tabular}
    }  
    \caption{List of 45 linear inequality constrained CUTEst problems used for numerical experiments.}
    \label{tab:pblincons}
\end{table}

Tables~\ref{tab:pbboundcons} and~\ref{tab:pblincons}
contain lists of the 77 bound-constrained and 45 general linear inequality constrained CUTEst problems \cite{Gould2015,Fowkes2022} used for the numerical experiments in Section~\ref{sec_numerics}, based on the problems used in \cite{Gratton2019}.
Any values in brackets after the problem name are the  optional problem parameters used.
The columns $n$, NB and LI are the problem dimension, number of (finite) bound constraints, and number of (finite) general linear inequality constraints for each problem respectively.

\section{Detailed Numerical Results} \label{app_detailed_numerics}
Here, we show the same numerical results as in Figure~\ref{fig_main_results}, but split separately into the 77 bound-constrained test problems and the 45 problems with general linear inequality constraints. 
For brevity, we only show data profiles.

\begin{figure}[H]
  \centering
  \begin{subfigure}[b]{0.35\textwidth}
    \includegraphics[width=\textwidth]{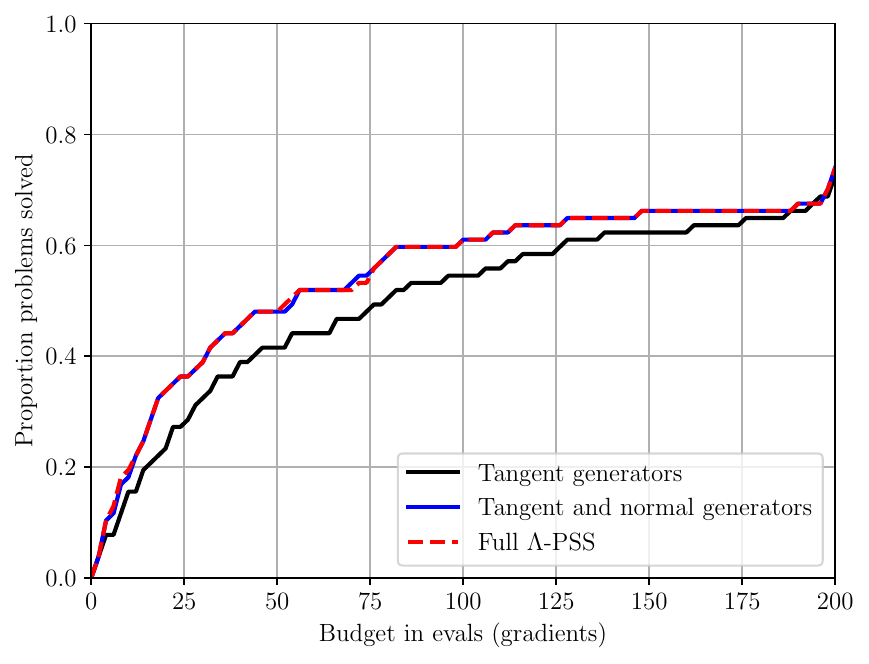}
    \caption{Bound constraints, $\tau=10^{-3}$}
  \end{subfigure}
  ~
  \begin{subfigure}[b]{0.35\textwidth}
    \includegraphics[width=\textwidth]{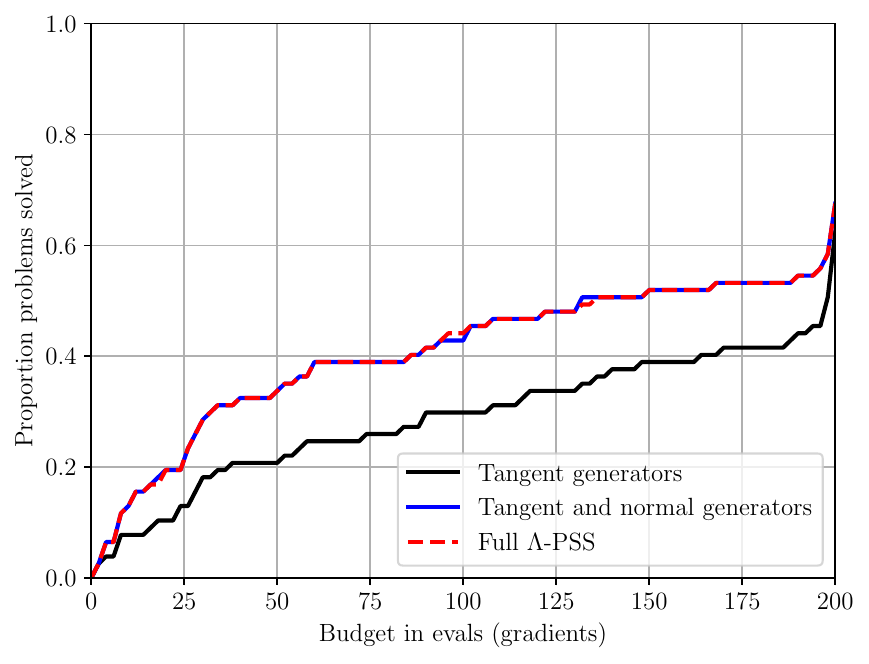}
    \caption{Bound constraints, $\tau=10^{-6}$}
  \end{subfigure}
  \\
  \begin{subfigure}[b]{0.35\textwidth}
    \includegraphics[width=\textwidth]{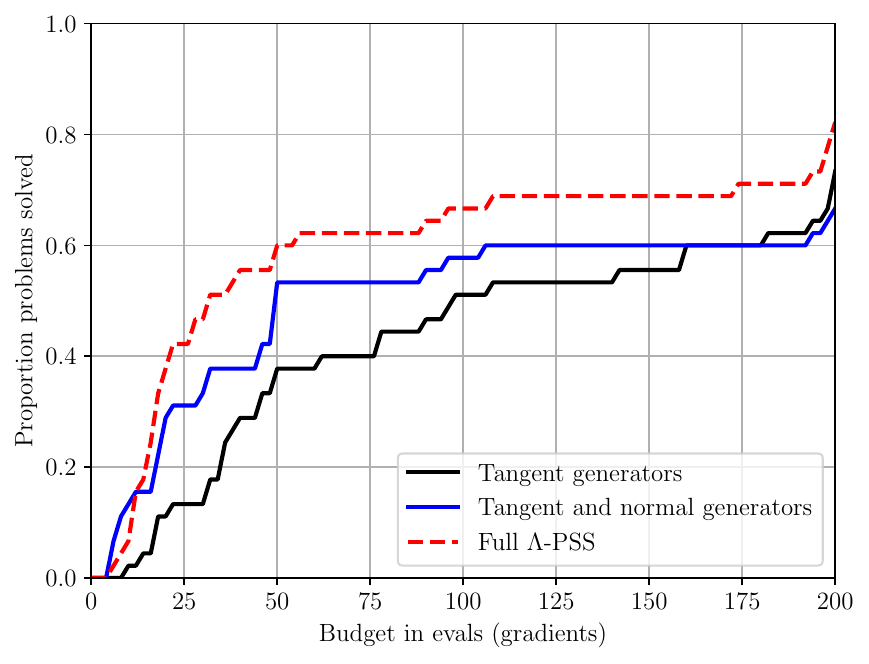}
    \caption{Linear ineq.~constraints, $\tau=10^{-3}$}
  \end{subfigure}
  ~
  \begin{subfigure}[b]{0.35\textwidth}
    \includegraphics[width=\textwidth]{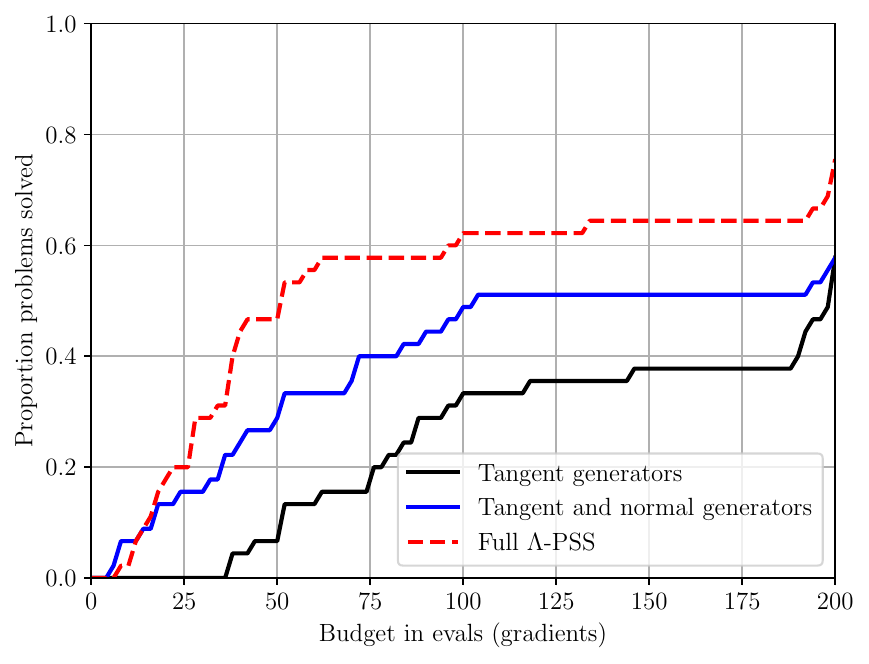}
    \caption{Linear ineq.~constraints, $\tau=10^{-6}$}
  \end{subfigure}
  \caption{Numerical results split by problem type.}
  \label{fig_detailed_numerics}
\end{figure}

\end{document}